\RequirePackage{etex}
\documentclass[a4paper, 12pt]{article}
\usepackage[pdfstartview=FitH,
            bookmarksnumbered=true,
            bookmarksopen=true,
            hidelinks
            ]{hyperref}
\usepackage[left=2.5cm, right=2.5cm, top=2.5cm, bottom=2.2cm]{geometry}
\usepackage{bm}
\usepackage{amsfonts,amssymb} 
\usepackage{amsmath} 
\usepackage{xcolor}
\usepackage{tikz}
\usepackage{overpic}
\usepackage[titletoc,title]{appendix}
\numberwithin{equation}{section}
\usepackage[noabbrev]{cleveref} 
\crefname{appsec}{appendix}{appendices} 
\usepackage[ruled,vlined,algosection]{algorithm2e}
\usepackage{algpseudocode}
\usepackage{mathtools}
\usepackage{commath}
\usepackage{booktabs}
\usepackage{siunitx} 
\usepackage{color}
\usepackage{stmaryrd}
\definecolor{dkgreen}{rgb}{0,0.6,0}
\definecolor{gray}{rgb}{0.5,0.5,0.5}
\definecolor{mauve}{rgb}{0.58,0,0.82}
\usepackage{listings}
\usepackage{amsthm}
\newtheorem{theorem}{Theorem}[section]

\newtheorem{example}{Example}[section]
\newtheorem{lemma}{Lemma}[section]
\newtheorem{corollary}{Corollary}[section]

\newtheorem{remark}{Remark}[section]

\DeclareMathOperator\erf{erf}
\DeclareMathOperator{\spn}{span}

\DeclareMathOperator{\spp}{supp}

\definecolor{myBrown}{rgb}{0.6 0.4 0.2}
\definecolor{myOrange}{rgb}{1.0 0.6 0.2}
\definecolor{myLightGray}{RGB}{235,235,235}
\definecolor{myViolet}{RGB}{153,50,204}
\newcommand\rev[1]{{\color{black}#1}}
\newcommand{\fx}[1]{{\color{black}#1}}
\newcommand{\bbN}{\mathbb{N}}
\newcommand{\fkJ}{\mathcal{I}}
\newcommand{\keywords}[1]{\textbf{Key words.} #1}
\newcommand{\AMS}[1]{\textbf{AMS subject classifications.} #1}

\title{A posteriori error estimation and adaptivity in stochastic Galerkin FEM
for parametric\\ elliptic PDEs: beyond the affine case\thanks{This work was supported by the EPSRC under grant EP/P013791/1
and by The Alan Turing Institute under the EPSRC grant EP/N510129/1.}}
\author{Alex Bespalov\footnotemark[2] \and Feng Xu\footnotemark[2]}

\begin{document}

\date{}
\maketitle

\long\def\symbolfootnote[#1]#2{\begingroup
\def\thefootnote{\fnsymbol{footnote}}\footnote[#1]{#2}\endgroup}

\renewcommand{\thefootnote}{\fnsymbol{footnote}}
\footnotetext[2]{School of Mathematics, University of Birmingham,
Edgbaston, Birmingham B15 2TT ({\tt a.bespalov@bham.ac.uk}, {\tt f.xu.2@bham.ac.uk}).}

\renewcommand{\thefootnote}{\arabic{footnote}}

\begin{abstract}
We consider a linear elliptic partial differential equation (PDE) with a generic uniformly bounded parametric coefficient.
The solution to this PDE problem is approximated in the framework of stochastic Galerkin finite element methods.
We perform a posteriori error analysis of Galerkin approximations and derive a reliable and efficient estimate
for the energy error in these approximations.
Practical versions of this error estimate are discussed and tested numerically for
a model problem with non-affine parametric representation of the coefficient.
Furthermore, we use the error reduction indicators
derived from spatial and parametric error estimators to guide an adaptive solution algorithm
for the given parametric PDE problem.
The performance of the adaptive algorithm is tested numerically for
model problems with two different non-affine parametric representations of the coefficient.
\end{abstract}

\keywords{stochastic Galerkin methods, stochastic finite element methods, parametric PDEs, a posteriori error estimation, adaptive methods, sparse polynomial approximation, generalized polynomial chaos expansion}

\AMS{35R60, 65C20, 65N30, 65N15}

\section{Introduction} \label{sec:intro}

Partial differential equations (PDEs) with uncertain or parameter-dependent inputs
arise in mathematical models of many physical phenomena as well as in engineering applications.
Stochastic Galerkin finite element method (sGFEM) is commonly used for solving
such PDE problems numerically, in particular, when the input data and solutions are sufficiently smooth
functions of parameters.
The sGFEM solution is sought in the tensor product of a finite element space defined on the physical domain and
a multivariable polynomial space on the parameter domain.
Even if a moderate number of parameters is used to represent the problem inputs, the cost
associated with computing high-fidelity sGFEM approximations quickly becomes prohibitive, due to
fast growth of the dimension of the tensor product space.
An adaptive approach to constructing approximation spaces provides a remedy to this computational bottleneck.
Based on rigorous a posteriori error analysis of computed solutions,
adaptive solution techniques build spatial and parametric components
of approximations incrementally in the course of numerical computation,
leading to accelerated convergence and reduced computational cost.

For elliptic PDE problems with \emph{affine-parametric} coefficients,
several adaptive sGFEM algorithms have been recently proposed and analyzed,
see, e.g., \cite{gittelson13,2014-Eigel-vol270,2015-Eigel-p1367,2016-Bespalov-vol38,em16,2018-Bespalov-p243,2019-Bespalov-p2359}.
A range of the underlying a posteriori error estimation techniques is used in these and other works
in order to guide adaptive refinement
(e.g., residual-based, local equilibration, and hierarchical a posteriori error estimators and error indicators to name but a few).
By contrast, 
the sGFEM-based numerical schemes for problems
with \emph{non-affine} parametric representations of coefficients are significantly less well developed.
As far as adaptive stochastic Galerkin approximations are concerned,
the only work we are aware of is~\cite{EigelMPS_ASG},
where the adaptive sGFEM procedure driven by reliable \emph{residual-based} error indicators
is developed for linear elliptic PDEs with lognormal coefficients.
It is worth noting that, due to unboundedness of coefficients,
a well-posed weak formulation of this problem needs to be introduced in problem-dependent weighted spaces,
as presented in~\cite{2011-Schwab-p291}.
Practical feasibility of the adaptive algorithm in~\cite{EigelMPS_ASG} is ensured by
adaptive discretizations of the lognormal coefficient represented in a hierarchical tensor format,
as described in~\cite{2017-Eigel-p765}, under the assumption that the errors in such discretizations are small.

In this paper, we consider a linear elliptic PDE with a \emph{generic} parametric coefficient.
\rev{In particular, our analysis is not restricted
to any specific form of the parametric coefficient (affine, quadratic, log-uniform, etc.).}
Assuming uniform boundedness of the coefficient,
which is a minimal requirement to ensure well-posedness of the weak formulation in standard Lebesgue--Bochner spaces,
we derive a reliable and efficient a posteriori estimate of the energy error in sGFEM approximations.
\rev{This extends the analysis of hierarchical error estimators presented in~\cite{2014-Bespalov-vol36,2016-Bespalov-vol38}
and fills a gap in the existing theory}.
Two practical examples of hierarchical error estimates are considered in detail
and studied numerically for the steady-state diffusion problem
with non-affine parametric representation of the coefficient.
We then present an adaptive algorithm driven by the error reduction indicators
derived from hierarchical a posteriori error estimators in the spirit of~\cite{2016-Bespalov-vol38,2018-Bespalov-p243}.
The performance of the adaptive algorithm is tested numerically for two non-affine parametric representations
of the diffusion coefficient.

The rest of the paper is structured as follows.
The model problem is introduced in section~\ref{sec:problem};
its Galerkin approximation and a posteriori error estimation are presented in section~\ref{sec:apprerr}.
The generalized polynomial chaos (gPC) expansion of the parametric coefficient and
the associated practical aspects of the developed error estimation strategy are discussed in section~\ref{sec:diffcoe},
while the results of numerical tests are reported in section~\ref{sec:numer:errest}.
The adaptive algorithm is proposed in section~\ref{sec:adaptive},
and its performance is tested in numerical experiments described in section~\ref{sec:numer:adapt}.
In Appendix~\ref{appA}, we derive explicit formulae for calculating the gPC expansion coefficients for
parametric exponential and quadratic functions.

\section{Stochastic steady-state diffusion problem} \label{sec:problem}

Let $D \subset \mathbb{R}^2$ be a bounded (spatial) domain with a Lipschitz polygonal boundary $\partial D$, and let
$\Gamma \coloneqq \prod_{m = 1}^{\fx{\infty}} \Gamma_m$ be
the parameter domain with bounded \rev{intervals} $\Gamma_m \subset \mathbb{R}$.
Let $H_0^1(D)$ be the usual Sobolev space of functions in $H^1(D)$ vanishing at the boundary $\partial D$ in the sense of traces. \fx{We will use the standard norm in $H_0^1(D)$ as $\norm{v}_{H_0^1(D)} \coloneqq \norm{\nabla v}_{L^2(D)}$.}
As an example of model problem, we consider the homogeneous Dirichlet problem for the parametric steady-state diffusion equation
\begin{equation}
\label{paraPDE}
\begin{aligned}
-\nabla \cdot (T(\bm{x}, \bm{y})\nabla u(\bm{x}, \bm{y})) &= f(\bm{x}),  &&\bm{x} \in D,  ~ \bm{y} = (y_1, y_2, \dots) \in \Gamma, \\
u(\bm{x}, \bm{y}) &= 0,  &&\bm{x} \in \partial D, ~ \bm{y} \in \Gamma,
\end{aligned}
\end{equation}
where $f \in H^{-1} (D)$ and $\nabla$ denotes differentiation with respect to $\bm{x}$ only.
We assume that the parameters $y_m$, \fx{$m \in \bbN$}, are the images of \emph{independent} random variables with cumulative \rev{distribution} function $\pi_m(y_m)$ and probability density function $q_m(y_m) = \dif \pi_m(y_m) /\dif y_m$. Then for the multivariate random variable formed by all independent univariate random variables, the joint cumulative \rev{distribution} function  and the joint probability density function are
$\pi (\bm{y}) \coloneqq \prod_{m = 1}^{\infty} \pi_m(y_m)$ and
$q(\bm{y}) \coloneqq \prod_{m = 1}^{\infty} q_m(y_m)$, respectively.
Since each $\Gamma_m$ is bounded, we can always rescale the corresponding univariate random variable such that it takes values in $[-1, 1]$.
Therefore, without loss of generality, we assume that $\Gamma_m \coloneqq [-1, 1]$ for all $m \in \bbN$.

Note that each $\pi_m$ is a  probability measure on $(\Gamma_m, \mathcal{B}(\Gamma_m))$,
where $\mathcal{B}(\Gamma_m)$ is the Borel $\sigma$-algebra on $\Gamma_m$.
Accordingly, $\pi$ is a probability measure on $(\Gamma, \mathcal{B}(\Gamma))$,
where $\mathcal{B}(\Gamma)$ is the Borel $\sigma$-algebra on $\Gamma$.
Then $L_{\pi_m}^2(\Gamma_m)$ (resp., $L_{\pi}^2(\Gamma)$) represents the Lebesgue space
of \rev{equivalence classes of} functions $v: \Gamma_m \to \mathbb{R}$ (resp., $v: \Gamma \to \mathbb{R}$)
that are square integrable on $\Gamma_m$ (resp.,~$\Gamma$)
with respect to the measure $\pi_m$ (resp., $\pi$), and
$\langle\cdot,\cdot\rangle_{\pi_m}$ (resp., $\langle\cdot,\cdot\rangle_{\pi}$) denotes the associated inner product:
$\langle f, g\rangle_{\pi_m} \coloneqq \int_{\Gamma_m} q_m(y_m) f(y_m) g(y_m) \dif y_m$ for $f, g \in L_{\pi_m}^2(\Gamma_m)$
(resp., $\langle f, g\rangle_{\pi} \coloneqq \int_{\Gamma} q(\bm{y}) f(\bm{y}) g(\bm{y}) \dif \bm{y}$ for $f, g \in L_{\pi}^2(\Gamma)$).
For a Hilbert space $H$ of functions on $D$, we will denote by $L_{\pi}^2(\Gamma; H)$ the space
of strongly measurable \rev{equivalence classes of} functions $v: D\times \Gamma \to \mathbb{R}$ such that
\begin{align*}
   \norm{v}_{\fx{L_{\pi}^2(\Gamma; H)}} \coloneqq \left(\fx{\int_{\Gamma} q(\bm{y}) \norm{v(\cdot, \bm{y})}_{H}^2 \dif \bm{y}}\right)^{1/2} < +\infty.
\end{align*}
In particular, we will denote
$V \coloneqq L_{\pi}^2(\Gamma; H_0^1(D))$ and $W \coloneqq L_{\pi}^2(\Gamma; L^2(D))$.

The weak formulation of (\ref{paraPDE}) reads as follows: find $u \in V$ such that
\begin{align}
\label{wf}
B(u, v) = F(v) \quad \forall v \in V,
\end{align}
where the symmetric bilinear form $B(\cdot, \cdot)$ and the linear functional $F(\cdot)$ are defined by
\begin{align}
\label{tbf}
B(u,v) &\coloneqq \int_{\Gamma}q(\bm{y})\int_D T(\bm{x}, \bm{y}) \nabla u(\bm{x}, \bm{y}) \cdot \nabla v(\bm{x}, \bm{y})\dif \bm{x} \dif \bm{y},  \\[4pt]
\label{lf}
F(v) &\coloneqq \int_{\Gamma}q(\bm{y})\int_D f(\bm{x}) v(\bm{x}, \bm{y})\dif \bm{x} \dif \bm{y}.
\end{align}

To ensure the well-posedness of (\ref{wf}), we make the following assumption on the parametric diffusion coefficient
\rev{$T \in L_{\pi}^{\infty}(\Gamma; L^{\infty}(D))$}:
there exist constants $\alpha_{\min}$  and $\alpha_{\max}$ such~that
\begin{align}
\label{boundT}
0<\alpha_{\min} \leq T(\bm{x}, \bm{y}) \leq \alpha_{\max} < \infty \quad \hbox{a.e. in $D \times \Gamma$}.
\end{align}
In particular, this implies that $B(\cdot,\cdot)$ is continuous and elliptic on $V$.
Therefore, 
$B(\cdot, \cdot)$ defines an inner product in $V$ which induces the norm
$\norm{v}_B \coloneqq B(v,v)^{1/2}$ that is equivalent to $\norm{v}_V$, i.e.,
\begin{align} \label{Bineq}
     \alpha_{\min} \norm{v}_V^2  \leq \fx{\norm{v}_B^2} \leq \alpha_{\max} \norm{v}_V^2  \quad \forall v \in V.
\end{align}

\section{Galerkin approximation and a posteriori error estimation} \label{sec:apprerr}

\subsection{Galerkin approximation}

Let us introduce the finite-dimensional approximation of the weak problem (\ref{wf}).
Problem~(\ref{wf}) can be discretized by using Galerkin projection onto any finite-dimensional subspace of $V$.
Note that the space $V = L_{\pi}^2(\Gamma; H_0^1(D))$
is isometrically isomorphic to the tensor product Hilbert space $H_0^1 (D)\otimes L_{\pi}^2(\Gamma)$
(see, e.g.,~\cite[Theorem~B.17, Remark~C.24]{2011-Schwab-p291}).
Hence we can construct the finite-dimensional subspace of $V$ by tensorizing a finite-dimensional subspace of $H_0^1 (D)$ and a finite-dimensional subspace of $L_{\pi}^2(\Gamma)$.

For the finite-dimensional subspace of $H_0^1 (D)$, we choose the finite element space
$X = \spn\{\phi_1, \dots, \phi_{n_X}\}$, where $\phi_i$ are standard finite element basis functions and $n_X = \dim(X)$.

Let us now introduce the finite-dimensional (polynomial) subspaces of $L_{\pi}^2(\Gamma)$.
To that end, we consider 
\fx{the following set of finitely supported sequences:
\begin{align*}
   \mathcal{I} \coloneqq
   \big\{\bm{\alpha} = (\alpha_1, \alpha_2, \cdots) \in \mathbb{N}_0^{\mathbb{N}};\; \max(\spp \bm{\alpha}) < \infty\big\},
\end{align*}
where $\spp \bm{\alpha} = \{m \in \mathbb{N};\; \alpha_m \neq 0\}$.
The set $\mathcal{I}$, as well as any of its subsets, will be called the \emph{index set},
and the elements $\bm{\alpha} \in \mathcal{I}$ will be called the \emph{(multi-)indices}.
For each $m \in \bbN$, let $\{p_n^m\}_{n\in\mathbb{N}_0}$ denote the set of univariate polynomials on $\Gamma_m$
that are orthonormal with respect to the inner product $\langle\cdot, \cdot\rangle_{\pi_m}$ in $L_{\pi_m}^2(\Gamma_m)$.
Then we can define the following tensor product polynomials:
\begin{align*}
p_{\bm{\alpha}}(\bm{y}) \coloneqq \prod_{m=1}^{\infty} p_{\alpha_m}^m (y_m)
   = \prod_{m \in \spp \bm{\alpha}} p_{\alpha_m}^m (y_m) \quad \forall \bm{\alpha} \in \mathcal{I}.
\end{align*}
The countable set $\{p_{\bm{\alpha}};\; \bm{\alpha} \in \fkJ\}$
forms an orthonormal basis of $L_{\pi}^2(\Gamma)$ (see, e.g., \cite[section~3.3]{2012-Ernst-p317})}.
Given a finite index set $\mathcal{P} \subset \fkJ$, the space of tensor product polynomials
$ P_\mathcal{P} \coloneqq \spn \left\{p_{\bm{\alpha}};\; \bm{\alpha} \in \mathcal{P}\right\}$
defines a finite-dimensional subspace of  $L_{\pi}^2(\Gamma)$.

With both spaces $X \subset H_0^1(D)$ and $ P_\mathcal{P} \subset L_{\pi}^2(\Gamma)$, we can now define the finite-dimensional subspace  $V_{X\mathcal{P}} \coloneqq X \otimes  P_\mathcal{P} \subset V$
and write the discrete formulation of (\ref{wf}) as follows:
find $u_{X\mathcal{P}} \in V_{X\mathcal{P}}$ such that
\begin{align}
\label{dwf}
B(u_{X\mathcal{P}}, v) = F(v) \quad \forall v \in V_{X\mathcal{P}}.
\end{align}
Hereafter, we assume that $\mathcal{P}$ always contains the zero-index $\bm{0} \coloneqq \fx{(0, 0, \dots)}$.

\subsection{A posteriori error estimation}

The aim of this subsection is to generalize the results of~\cite{2016-Bespalov-vol38} to
the case of the diffusion coefficient $T(\bm{x}, \bm{y})$ satisfying only the boundedness assumption~\eqref{boundT}
(which is a minimal assumption that guarantees the well-posedness of the weak formulation~\eqref{wf}).

We follow the classical hierarchical a posteriori error estimation strategy as described, e.g., in~\cite[Chapter~5]{Ainsworth2000}.
First, let us briefly outline the main ingredients of this strategy emphasizing the specific features pertaining
to tensor-product approximations.
The starting point is the following equation for the discretization error
$e \,{\coloneqq}\, u - u_{X\mathcal{P}} \,{\in}\, V$:
\begin{align} \label{weak_error}
B(e, v) = F(v) - B(u_{X\mathcal{P}}, v)  \quad \forall v \in V.
\end{align}
Since $e$ lives in the infinite-dimensional space $V$, we cannot calculate $e$ by using (\ref{weak_error}) directly.
However, one can approximate the error $e$ in  a finite-dimensional subspace $V_{X\mathcal{P}}^* \subset V$
in a similar way as the solution $u$ is approximated in the finite-dimensional subspace $V_{X\mathcal{P}} \subset V$.
Specifically, we introduce the error estimator $e^* \in V_{X\mathcal{P}}^*$ that~satisfies
\begin{align} \label{firstestimator}
     B(e^*, v) = F(v) - B(u_{X\mathcal{P}}, v) \quad \forall v \in V_{X\mathcal{P}}^*.
\end{align}
Note that, due to Galerkin orthogonality
\begin{align}\label{gal_orth}
B(e, v) = F(v) - B(u_{X\mathcal{P}}, v) = 0 \quad \forall v \in V_{X\mathcal{P}},
\end{align}
a meaningful approximation of $e$ is obtained by requiring that $V_{X\mathcal{P}} \subsetneqq V_{X\mathcal{P}} ^*$.

It is well known that the error estimator $e^*$ is linked to the 
enhanced Galerkin approximation $u_{X\mathcal{P}}^* \in V_{X\mathcal{P}}^*$ as follows:
$e^* = u_{X\mathcal{P}}^* - u_{X\mathcal{P}}$.
Here, $u_{X\mathcal{P}}^* \in V_{X\mathcal{P}}^*$ satisfies
\begin{align}
\label{dwfenhanced}
B(u_{X\mathcal{P}}^*, v) = F(v) \quad \forall v \in V_{X\mathcal{P}}^*.
\end{align}
Furthermore, since $B(\cdot,\cdot)$ is symmetric,
we deduce from \eqref{wf}, \eqref{dwf}, \eqref{dwfenhanced} that
\begin{align} \label{err:norms:aux1}
   \norm{e}_B^2 = \norm{u - u_{X\mathcal{P}}}_B^2 = F(u) - F(u_{X\mathcal{P}}), \qquad
   \norm{u - u_{X\mathcal{P}}^*}_B^2 = F(u) - F(u_{X\mathcal{P}}^*)
\end{align}
and
\begin{align*} 
   \norm{e^* }_B^2 = \norm{u_{X\mathcal{P}}^* - u_{X\mathcal{P}}}_B^2 =
                                 F(u_{X\mathcal{P}}^*) - F(u_{X\mathcal{P}}) \stackrel{\eqref{err:norms:aux1}}{=}
                                 \norm{e}_B^2 - \norm{u - u_{X\mathcal{P}}^*}_B^2.
\end{align*}
This implies that:
(i) $\norm{e^*}_B \leq \norm{e}_B$;
(ii) the quantity $\norm{e^* }_B$ is
the energy error reduction achieved by using the enriched space $V_{X\mathcal{P}}^*$;
and
(iii) $\|u - u_{X \mathcal{P}}^*\|_B \leq \|u - u_{X \mathcal{P}}\|_B$.
In order to establish the equivalence between the true energy error $\norm{e}_B$
and the energy error estimate $\norm{e^*}_B$,
the following stronger property than the one given in (iii) is assumed
(this property is usually referred to as the saturation assumption):
there exists a constant $\beta \in [0, 1)$ such~that
\begin{align}
\label{saturation}
\norm{u - u_{X \mathcal{P}}^*}_B \leq \beta \norm{u - u_{X \mathcal{P}}}_B.
\end{align}
Then the following inequalities hold (see, e.g., \cite[Theorem 5.1]{Ainsworth2000}):
\begin{align}
\label{eeStar}
   \norm{e^*}_B \leq \norm{e}_B \leq \frac{1}{\sqrt{1-\beta^2}}\norm {e^*}_B.
\end{align}

Motivated by high computational cost involved in computing the error estimator $e^*$ defined by (\ref{firstestimator})
(the cost that is comparable to computing the enhanced Galerkin approximation $u_{X\mathcal{P}}^*$),
hierarchical a posteriori error estimation techniques seek to approximate $e^*$ by
making use of the following two key ingredients:
(a) an alternative bilinear form $\widetilde{B}(\cdot, \cdot)$ in place of $B(\cdot, \cdot)$ on the left-hand side in (\ref{firstestimator})
with the aim to obtain an easier to invert (stiffness) matrix in the associated linear system;
(b) an appropriate decomposition of the enhanced finite-dimensional space $V_{X\mathcal{P}}^*$
with the aim to further reduce computational cost by solving (\ref{firstestimator}) on the subspace(s) of $V_{X\mathcal{P}}^*$.

The alternative bilinear form $\widetilde{B}(\cdot, \cdot)$ is employed to define
the modified error estimator $\tilde{e} \in V_{X\mathcal{P}}^*$ satisfying
\begin{align} \label{secondestimator}
\widetilde{B}(\tilde{e}, v) = F(v) - B(u_{X\mathcal{P}}, v) \quad \forall v \in V_{X\mathcal{P}}^*.
\end{align}
For problem \eqref{secondestimator} to be well-posed, the auxiliary bilinear form $\widetilde{B}(\cdot, \cdot)$
is assumed to be \fx{symmetric, continuous, and elliptic}.
In this case, $\widetilde{B}$ defines an inner product in $V$ which induces the norm
$\norm{v}_{\widetilde{B}} \coloneqq \widetilde{B}(v, v)^{1/2}$ that is equivalent to $\norm{v}_B$, i.e.,
there exist two positive constants $\lambda$ and $\Lambda$ such that
\begin{align}
\label{normequiv}
\lambda \norm{v}_B \leq \norm{v}_{\widetilde{B}} \leq \Lambda \norm{v}_B \quad \forall v \in V.
\end{align}
This leads to the following relation between the error estimators $e^*$ and $\tilde{e}$
(see, e.g., \cite[Theorem 5.3]{Ainsworth2000}):
\begin{align}
\label{eSeT}
\lambda \norm{\tilde{e}}_{\widetilde{B}} \leq \norm{e^*}_B \leq \Lambda \norm{\tilde{e}}_{\widetilde{B}}.
\end{align}

The discussion of the second ingredient of the hierarchical error estimation strategy is linked
to the specific choice of the enriched subspace $V_{X\mathcal{P}} ^* \subset V$.
In the context of tensor-product approximations, an appropriate choice of $V_{X\mathcal{P}} ^*$ is important,
as this affects the quality of the final error estimate as well as the computational cost associated
with computing that estimate, cf.~\cite{2014-Bespalov-vol36,2016-Bespalov-vol38}.
In this paper, we follow the idea proposed in~\cite{2016-Bespalov-vol38}.
Firstly, we construct an enriched finite element subspace $X^* \subset H_0^1(D)$, which has a direct sum decomposition
$X^* \coloneqq X \oplus Y$,
where the finite-dimensional subspace $Y\subset H_0^1(D)$ is called the \emph{detail finite element space}.
Secondly, we construct an enriched polynomial space
$ P_{\mathcal{P}^*} \coloneqq \spn\{p_{\bm{\alpha}};\; \bm{\alpha} \in {\mathcal{P}^*}\}$ associated with
a finite index set $\mathcal{P}^* \coloneqq \mathcal{P} \cup \mathcal{Q}$ for some
$\mathcal{Q} \subset \fkJ$ such that $\mathcal{P} \cap \mathcal{Q} = \emptyset$.
The set $\mathcal{Q}$ is called the \emph{detail index set} and the corresponding polynomial space
$ P_{\mathcal{Q}} \coloneqq \spn\{p_{\bm{\alpha}};\; \bm{\alpha} \in {\mathcal{Q}}\}$ is called the \emph{detail polynomial space}.
Note that $ P_{\mathcal{P}^*}$ has an orthogonal direct sum decomposition
with respect to the inner product $\langle\cdot, \cdot\rangle_{\pi}$ as follows:
\begin{align*}
    P_{\mathcal{P}^*} =  P_{\mathcal{P}} \oplus  P_{\mathcal{Q}}.
\end{align*}
Finally, the enriched  finite-dimensional space $V_{X\mathcal{P}}^*$ is defined as the following direct~sum:
\begin{align*} 
   V_{X\mathcal{P}}^* \coloneqq V_{X\mathcal{P}} \oplus V_{Y\mathcal{P}} \oplus V_{X\mathcal{Q}},
\end{align*}
where $V_{Y\mathcal{P}} \coloneqq Y \otimes  P_\mathcal{P}$ and
$V_{X\mathcal{Q}} \coloneqq X \otimes  P_\mathcal{Q}$.

The direct sum structure of $V_{X\mathcal{P}}^*$ motivates
the definition of two error estimators $e_{Y\mathcal{P}} \in V_{Y\mathcal{P}}$ and
$e_{X\mathcal{Q}} \in V_{X\mathcal{Q}}$ satisfying
\begin{align}
\label{eYPestimator}
\widetilde{B}(e_{Y\mathcal{P}}, v) &= F(v) - B(u_{X\mathcal{P}}, v) && \forall v \in V_{Y\mathcal{P}}, \\
\label{eXQestimator}
\widetilde{B}(e_{X\mathcal{Q}}, v) &= F(v) - B(u_{X\mathcal{P}}, v) && \forall v \in V_{X\mathcal{Q}}.
\end{align}

Combining all ingredients,
we define the following error estimate
\begin{align}
\label{estimate}
\eta \coloneqq \sqrt{\norm{e_{Y \mathcal{P}}}_{\widetilde{B}}^2  + \norm{e_{X \mathcal{Q}}}_{\widetilde{B}}^2}.
\end{align}

Clearly, making the right choice of the auxiliary bilinear form $\widetilde{B}(\cdot,\cdot)$ is important in the above construction.
In particular, if this choice implies $\widetilde{B}$-orthogonality of the subspace decomposition,
the following abstract result holds.

\begin{lemma} \label{lem1}
Let $\widetilde{B}(\cdot,\cdot)$ be a symmetric bilinear form that is continuous and elliptic on a Hilbert space $V$,
and let $G(\cdot)$ be a continuous linear functional on $V$.
Consider three subspaces $V_1,\, V_2,\, V_3 \subset V$ such that $V_3 = V_1 \oplus V_2$.
Let $e_i \in V_i$ ($i = 1,\,2,\,3$) satisfy
\begin{align} \label{eq:ei}
   \widetilde{B}(e_i, v) = G(v)\qquad \forall v \in V_i.
\end{align}
If the direct sum decomposition $V_3 = V_1 \oplus V_2$ is $\widetilde{B}$-orthogonal, i.e.,
\begin{align} \label{orgen}
   \widetilde{B} (u, v) = 0 \qquad \forall u \in V_1, ~~ \forall v \in V_2,
\end{align}
then
\begin{align} \label{eq:lem1}
   e_3 = e_1 + e_2 \quad \hbox{and}\quad \widetilde{B} (e_3, e_3) = \widetilde{B} (e_1, e_1) + \widetilde{B} (e_2, e_2).
\end{align}
\end{lemma}

\begin{proof}
Since $V_3 = V_1 \oplus V_2$, every $v \in V_3$ has a unique decomposition
$v = v_1 + v_2\text{ with }  v_1 \in V_1, ~ v_2\in V_2$.
Using the orthogonality relation (\ref{orgen}), we deduce from (\ref{eq:ei}) (with $i = 1,\,2$) that
\begin{align*}
   \widetilde{B}(e_1 + e_2,v)
   &= \widetilde{B}(e_1 + e_2, v_1 + v_2) = \widetilde{B}(e_1, v_1) + \widetilde{B}(e_1, v_2) + \widetilde{B}(e_2, v_1) + \widetilde{B}(e_2, v_2)
   \\
   &= \widetilde{B}(e_1, v_1) + \widetilde{B}(e_2, v_2) = G(v_1) + G(v_2) = G(v)\quad \forall v \in V_3.
\end{align*}
This implies that $e_3=  e_1 + e_2$,
because $e_3$ is the unique solution of (\ref{eq:ei}) with $i=3$.
The second equality in~\eqref{eq:lem1} then follows due to the orthogonality property~\eqref{orgen}
and the symmetry of the bilinear form $\widetilde{B}(\cdot,\cdot)$.
\end{proof}

A direct application of Lemma~\ref{lem1} to the subspace
$V_{X\mathcal{Q}} = \oplus_{\bm{\mu}\in \mathcal{Q}} X \otimes  P_{\{\bm{\mu}\}}$
gives the following result on the decomposition of the error estimator $e_{X\mathcal{Q}}$ defined by~(\ref{eXQestimator}).

\begin{corollary} \label{col1}
Assume that the direct sum decomposition
$V_{X\mathcal{Q}} = \oplus_{\bm{\mu}\in \mathcal{Q}} X \otimes  P_{\{\bm{\mu}\}}$
is $\widetilde{B}$-orthogonal, i.e.,
for any $\bm{\mu},\,\bm{\nu} \in \mathcal{Q}$ \rev{($\bm{\mu} \neq \bm{\nu}$)} there holds
\begin{align} \label{orXQ}
   \widetilde{B}(u, v) = 0 \qquad \forall u \in X\otimes  P_{\{\bm{\mu}\}}, ~~ \forall v \in X\otimes  P_{\{\bm{\nu}\}}.
\end{align}
Then the error estimator $e_{X \mathcal{Q}}$ defined by~\eqref{eXQestimator}
and its norm $\norm{e_{X \mathcal{Q}}}_{\widetilde{B}}$ can be decomposed into the
contributions associated with individual indices $\bm{\mu} \in \mathcal{Q}$ as follows:
\begin{align} \label{eXQdecomp}
   e_{X \mathcal{Q}} = \sum_{\bm{\mu} \in \mathcal{Q}} e^{(\bm{\mu})}_{X \mathcal{Q}},\qquad
   \norm{e_{X \mathcal{Q}}}_{\widetilde{B}}^2 = \sum_{\bm{\mu} \in \mathcal{Q}} \norm{e^{(\bm{\mu})}_{X \mathcal{Q}}}_{\widetilde{B}}^2.
\end{align}
Here, for each index $\bm{\mu} \in \mathcal{Q}$,
the estimator $e^{(\bm{\mu})}_{X \mathcal{Q}} \in X \otimes  P_{\{\bm{\mu}\}}$ satisfies
\begin{align} \label{eXQestimatori}
   \widetilde{B}\Big(e^{(\bm{\mu})}_{X \mathcal{Q}}, v\Big) = F(v) - B(u_{X\mathcal{P}}, v)\qquad
   \forall v \in X \otimes  P_{\{\bm{\mu}\}}.
\end{align}
\end{corollary}

The next step is to connect the error estimates $\norm{\tilde{e}}_{\widetilde{B}}$ and $\eta$.
To this end, we employ two strengthened Cauchy--Schwarz inequalities \rev{(see, e.g.,~\cite{Ainsworth2000,eijkhoutVass1991})}:
there exist two constants $\kappa_1,\, \kappa_2 \in [0, 1)$ such~that
\begin{align}
\label{sCS1}
|\widetilde{B}(u, v)| &\leq \kappa_1 \norm{u}_{\widetilde{B}} \norm{v}_{\widetilde{B}} && \forall u \in V_{X^*\mathcal{P}}\coloneqq V_{X\mathcal{P}} \oplus V_{Y\mathcal{P}}, ~ \forall v \in V_{X\mathcal{Q}}, \\
\label{sCS2}
|\widetilde{B}(u, v)| &\leq \kappa_2 \norm{u}_{\widetilde{B}} \norm{v}_{\widetilde{B}} && \forall u \in V_{X\mathcal{P}}, ~ \forall v \in V_{Y\mathcal{P}}.
\end{align}

\begin{lemma} \label{the:first}
Let $\norm{\tilde{e}}_{\widetilde{B}}$ and $\eta$ be defined in \eqref{secondestimator} and~\eqref{estimate}, respectively.
Then the following inequalities hold
\begin{align} \label{errest}
   \frac{1 }{\sqrt{2}}\eta \leq \norm{\tilde{e}}_{\widetilde{B}} \leq \frac{1 }{\sqrt{(1-\kappa_1)(1-\kappa_2^2)}}\eta,
\end{align}
Furthermore, if $\kappa_1 = 0$ in \eqref{sCS1}
(that is, $V_{X^*\mathcal{P}}$ and $V_{X\mathcal{Q}}$ are ${\widetilde{B}}$-orthogonal), then
\begin{align} \label{errest2}
   \eta  \leq \norm{\tilde{e}}_{\widetilde{B}} \leq \frac{1 }{\sqrt{1-\kappa_2^2}}\eta.
\end{align} 
\end{lemma}

\begin{proof}
We start by defining an auxiliary error estimator $e_{X^*\mathcal{P}} \in V_{X^*\mathcal{P}}$ satisfying
\begin{align} \label{eXStarPestimator}
   \widetilde{B}(e_{X^*\mathcal{P}}, v) = F(v) - B(u_{X\mathcal{P}}, v) \quad \forall v \in V_{X^*\mathcal{P}}.
\end{align}
The proof then consists of four steps.

{\bf Step 1.}
In this step, we will establish the following inequalities:
\begin{align} \label{errlevel1}
   \frac{\norm{e_{X^* \mathcal{P}}}_{\widetilde{B}}^2  + \norm{e_{X \mathcal{Q}}}_{\widetilde{B}}^2}{2} \leq
   \norm{\tilde{e}}_{\widetilde{B}}^2 \leq
   \frac{\norm{e_{X^* \mathcal{P}}}_{\widetilde{B}}^2  + \norm{e_{X \mathcal{Q}}}_{\widetilde{B}}^2}{1-\kappa_1}.
\end{align}
Since $V_{X^*\mathcal{P}}$ and $V_{X\mathcal{Q}}$ are subspaces of $V_{X\mathcal{P}}^*$,
we use~(\ref{secondestimator}), (\ref{eXStarPestimator}), (\ref{eXQestimator}) and apply the Cauchy--Schwarz inequality to obtain
\begin{align*}
\norm{e_{X^* \mathcal{P}}}_{\widetilde{B}}^2 &= \widetilde{B}(e_{X^* \mathcal{P}}, e_{X^* \mathcal{P}})  = \widetilde{B}(\tilde{e}, e_{X^* \mathcal{P}}) \leq \norm{\tilde{e}}_{\widetilde{B}}\norm{e_{X^* \mathcal{P}}}_{\widetilde{B}}, \\
\norm{e_{X \mathcal{Q}}}_{\widetilde{B}}^2 &= \widetilde{B}(e_{X \mathcal{Q}}, e_{X \mathcal{Q}})  = \widetilde{B}(\tilde{e}, e_{X \mathcal{Q}}) \leq  \norm{\tilde{e}}_{\widetilde{B}}\norm{e_{X \mathcal{Q}}}_{\widetilde{B}}.
\end{align*}
Hence, the left-hand inequality in (\ref{errlevel1}) follows.

Let us now prove the right-hand inequality in (\ref{errlevel1}).
Since $V_{X \mathcal{P}}^* = V_{X^* \mathcal{P}} \oplus V_{X \mathcal{Q}}$, the estimator $\tilde{e} \in V_{X\mathcal{P}}^*$ has a unique decomposition
\begin{align*}
   \tilde{e} = w_{X^* \mathcal{P}} + w_{X \mathcal{Q}}\quad
   \text{with}~~ w_{X^* \mathcal{P}}  \in V_{X^* \mathcal{P}}, ~~ w_{X \mathcal{Q}} \in V_{X \mathcal{Q}}.
\end{align*}
Using this representation of $\tilde{e}$, we deduce that
\begin{align} \label{eStarineq1}
   \norm{\tilde{e}}_{\widetilde{B}}^2 &=\widetilde{B}(\tilde{e},\tilde{e}) =
   \widetilde{B}(\tilde{e},w_{X^* \mathcal{P}} + w_{X \mathcal{Q}}) =
   \widetilde{B}(\tilde{e},w_{X^* \mathcal{P}}) + \widetilde{B}(\tilde{e}, w_{X \mathcal{Q}}) \nonumber \\
   & =
   \widetilde{B}(e_{X^* \mathcal{P}},w_{X^* \mathcal{P}}) + \widetilde{B}(e_{X \mathcal{Q}}, w_{X \mathcal{Q}}) \leq
   \norm{e_{X^* \mathcal{P}}}_{\widetilde{B}}\norm{w_{X^* \mathcal{P}}}_{\widetilde{B}} +
   \norm{e_{X \mathcal{Q}}}_{\widetilde{B}} \norm{w_{X \mathcal{Q}}}_{\widetilde{B}}\nonumber \\
   & \leq
   \left(\norm{e_{X^* \mathcal{P}}}_{\widetilde{B}}^2  + \norm{e_{X \mathcal{Q}}}_{\widetilde{B}}^2 \right)^{1/2}
   \left(\norm{w_{X^* \mathcal{P}}}_{\widetilde{B}}^2 + \norm{w_{X \mathcal{Q}}}_{\widetilde{B}}^2\right)^{1/2},
\end{align}
where the fourth equality is due to (\ref{secondestimator}), (\ref{eXStarPestimator}) and (\ref{eXQestimator}),
the first inequality is due to the Cauchy--Schwarz inequality, and
the second inequality is due to the algebraic inequality $ab + cd \leq (a^2+c^2)^{1/2}(b^2+d^2)^{1/2}$.

On the other hand, we can estimate $\norm{\tilde{e}}_{\widetilde{B}}^2$ from below as follows:
\begin{align} \label{eStarineq2}
   \norm{\tilde{e}}_{\widetilde{B}}^2 &= \widetilde{B}(\tilde{e},\tilde{e}) =
   \widetilde{B}(w_{X^* \mathcal{P}} + w_{X \mathcal{Q}},w_{X^* \mathcal{P}} + w_{X \mathcal{Q}}) \nonumber \\
   & =
   \widetilde{B}(w_{X^* \mathcal{P}} ,w_{X^* \mathcal{P}} ) +
   2\widetilde{B}(w_{X^* \mathcal{P}} , w_{X \mathcal{Q}}) + \widetilde{B}(w_{X \mathcal{Q}}, w_{X \mathcal{Q}}) \nonumber \\
   & \!\!\!\stackrel{\eqref{sCS1}}{\geq}
   \norm{w_{X^* \mathcal{P}}}_{\widetilde{B}}^2 - 2 \kappa_1 
   \norm{w_{X^* \mathcal{P}}}_{\widetilde{B}} \norm{w_{X \mathcal{Q}}}_{\widetilde{B}}+ \norm{w_{X \mathcal{Q}}}_{\widetilde{B}}^2 \nonumber \\
   & \geq
   \norm{w_{X^* \mathcal{P}}}_{\widetilde{B}}^2 -\kappa_1 
   \left(\norm{w_{X^* \mathcal{P}}}_{\widetilde{B}}^2 + \norm{w_{X \mathcal{Q}}}_{\widetilde{B}}^2\right) +
   \norm{w_{X \mathcal{Q}}}_{\widetilde{B}}^2 \nonumber \\
   & =
   (1-\kappa_1)\left(\norm{w_{X^* \mathcal{P}}}_{\widetilde{B}}^2 + \norm{w_{X \mathcal{Q}}}_{\widetilde{B}}^2\right).
\end{align}
Combining (\ref{eStarineq1}) with (\ref{eStarineq2}) gives the right-hand inequality in (\ref{errlevel1}).

{\bf Step 2.}
In the second step, we will establish the following inequalities:
\begin{align} \label{errlevel2}
   \norm{e_{Y \mathcal{P}}}_{\widetilde{B}}^2  \leq
   \norm{e_{X^* \mathcal{P}}}_{\widetilde{B}}^2 \leq \frac{\norm{e_{Y \mathcal{P}}}_{\widetilde{B}}^2}{1-\kappa_2^2}.
\end{align}
Since $V_{Y\mathcal{P}} \subset V_{X^*\mathcal{P}}$, we use~(\ref{eXStarPestimator}), (\ref{eYPestimator})
and the Cauchy--Schwarz inequality to~obtain
\begin{align*}
   \norm{e_{Y \mathcal{P}}}_{\widetilde{B}}^2 = \widetilde{B}(e_{Y \mathcal{P}},e_{Y \mathcal{P}}) =
   \widetilde{B}(e_{X^* \mathcal{P}},e_{Y \mathcal{P}}) \leq \norm{e_{X^* \mathcal{P}}}_{\widetilde{B}}\norm{e_{Y \mathcal{P}}}_{\widetilde{B}}.
\end{align*}
Hence, the left-hand inequality in (\ref{errlevel2}) follows.

Using similar arguments as in Step 1, the proof for the right-hand inequality in (\ref{errlevel2})
first makes use of the decomposition
\begin{align*}
   V_{X^* \mathcal{P}} \ni e_{X^* \mathcal{P}} =
   w_{X \mathcal{P}} + w_{Y \mathcal{P}}\quad
   \text{with} ~~w_{X \mathcal{P}}  \in V_{X \mathcal{P}}, ~~w_{Y \mathcal{P}} \in V_{Y \mathcal{P}}
\end{align*}
and the Cauchy--Schwarz inequality to estimate
\begin{align} \label{eXStarPineq1}
   \norm{e_{X^* \mathcal{P}}}_{\widetilde{B}}^2
   &=
   \widetilde{B}(e_{X^* \mathcal{P}}, e_{X^* \mathcal{P}}) =
   \widetilde{B}(e_{X^* \mathcal{P}}, w_{X \mathcal{P}} + w_{Y \mathcal{P}}) \nonumber \\
   &=
   \widetilde{B}(e_{X^* \mathcal{P}}, w_{Y \mathcal{P}})  =  \widetilde{B}(e_{Y \mathcal{P}}, w_{Y \mathcal{P}}) \leq
   \norm{e_{Y \mathcal{P}}}_{\widetilde{B}}\norm{w_{Y \mathcal{P}}}_{\widetilde{B}};
\end{align}
here, the third equality is due to $\widetilde{B}(e_{X^* \mathcal{P}}, w_{X \mathcal{P}}) = 0$
as follows from (\ref{dwf}) and (\ref{eXStarPestimator}), and
the fourth equality is due to (\ref{eXStarPestimator}) and (\ref{eYPestimator}).
On the other hand, applying the strengthened Cauchy--Schwarz inequality (\ref{sCS2}) and
the algebraic inequality $2\kappa_2ab \leq a^2 + \kappa_2^2b^2$,
we obtain the lower bound for $\norm{e_{X^* \mathcal{P}}}_{\widetilde{B}}^2$:
\begin{align} \label{eXStarPineq2}
   \norm{e_{X^* \mathcal{P}}}_{\widetilde{B}}^2
   &=
   \widetilde{B}(e_{X^* \mathcal{P}}, e_{X^* \mathcal{P}}) =
   \widetilde{B}(w_{X\mathcal{P}} + w_{Y\mathcal{P}},w_{X\mathcal{P}} + w_{Y\mathcal{P}}) \nonumber \\
   & \geq
   \norm{w_{X \mathcal{P}}}_{\widetilde{B}}^2 - 2 \kappa_2  \norm{w_{X \mathcal{P}}}_{\widetilde{B}}
   \norm{w_{Y \mathcal{P}}}_{\widetilde{B}}+ \norm{w_{Y \mathcal{P}}}_{\widetilde{B}}^2 \nonumber \\
   & \geq
   \norm{w_{X \mathcal{P}}}_{\widetilde{B}}^2  -
   \norm{w_{X \mathcal{P}}}_{\widetilde{B}}^2-
   \kappa_2^2 \norm{w_{Y \mathcal{P}}}_{\widetilde{B}}^2 +
   \norm{w_{Y \mathcal{P}}}_{\widetilde{B}}^2 =
   (1-\kappa_2^2)\norm{w_{Y \mathcal{P}}}_{\widetilde{B}}^2.
\end{align}
Combining (\ref{eXStarPineq1}) with (\ref{eXStarPineq2}) gives the right-hand inequality in (\ref{errlevel2}).

{\bf Step 3.}
Combining (\ref{errlevel1}) with (\ref{errlevel2}) and recalling the definition of $\eta$
gives~(\ref{errest}).

{\bf Step 4 ($\kappa_1 = 0$).}
A tighter lower bound in~(\ref{errest}) can be proved in this case.
Indeed, using the ${\widetilde{B}}$-orthogonality of the decomposition
$V_{X \mathcal{P}}^* = V_{X^* \mathcal{P}} \oplus V_{X \mathcal{Q}}$ and applying Lemma~\ref{lem1}
we conclude that
$\norm{\tilde{e}}_{\widetilde{B}}^2 = \norm{e_{X^* \mathcal{P}}}_{\widetilde{B}}^2  + \norm{e_{X \mathcal{Q}}}_{\widetilde{B}}^2$.
%
Combining this equality with the estimates~(\ref{errlevel2}) from Step~2 and
recalling the definition of $\eta$
we obtain (\ref{errest2}).
\end{proof}

Putting together (\ref{eeStar}), (\ref{eSeT}), (\ref{errest}) and (\ref{errest2}),
the following theorem gives two-sided bounds for the energy norm (i.e., $B$-norm)
of the true discretization error $e = u - u_{X \mathcal{P}}$ in terms of the estimate $\eta$.

\begin{theorem} \label{the:second}
Let $u \in V$ be the solution of \eqref{wf} and
let $u_{X \mathcal{P}} \in V_{X \mathcal{P}}$ be the Galerkin approximation satisfying \eqref{dwf}.
Suppose that the saturation assumption \eqref{saturation} and the norm equivalence~\eqref{normequiv} hold.
Then the a posteriori error estimate $\eta$ defined by \eqref{estimate}~satisfies
\begin{align} \label{t2errest}
   \frac{\lambda }{\sqrt{2}} \eta \leq
   \norm{u - u_{X \mathcal{P}}}_B \leq
   \frac{\Lambda }{\sqrt{1-\beta^2}\sqrt{(1-\kappa_1)(1-\kappa_2^2)}}\eta,
\end{align}
where $\beta \in [0,1)$ is the constant in \eqref{saturation},
$\lambda$ and $\Lambda$ are the constants in \eqref{normequiv}, and
$\kappa_1,\, \kappa_2 \in [0,1)$ are the constants in the strengthened Cauchy--Schwarz inequalities \eqref{sCS1},~\eqref{sCS2}.

Furthermore, 
if $\kappa_1 = 0$ in \eqref{sCS1} (that is, $V_{X^*\mathcal{P}}$ and $V_{X\mathcal{Q}}$ are ${\widetilde{B}}$-orthogonal), then
\begin{align} \label{t2errest2}
\lambda \eta  \leq \norm{u - u_{X \mathcal{P}}}_B \leq \frac{\Lambda }{\sqrt{1-\beta^2}\sqrt{1-\kappa_2^2}}\eta.
\end{align}
\end{theorem}

\begin{remark} \label{rem:err:est}
While error estimates~\eqref{t2errest} are new,
the estimates in~\eqref{t2errest2} have been proved in~\cite[Theorem~4.1]{2016-Bespalov-vol38}
for the model problem~\eqref{paraPDE} with
the diffusion coefficient $T(\bm{x}, \bm{y})$ that has affine dependence on random parameters.
In that framework, the auxiliary bilinear form ${\widetilde{B}}(\cdot,\cdot)$ is associated
with the parameter-free part of the representation for $T(\bm{x}, \bm{y})$ and yields
the orthogonality of the decomposition
$V_{X \mathcal{P}}^* = V_{X^* \mathcal{P}} \oplus V_{X \mathcal{Q}}$.
Thus, Theorem~\ref{the:second} generalizes the results of~\cite{2016-Bespalov-vol38}
to the case of a more general diffusion coefficient $T(\bm{x}, \bm{y})$ that is only assumed to be bounded
(the assumption that ensures the well-posedness of~\eqref{wf}).
In fact, our result is not limited to the diffusion problem~\eqref{paraPDE}.
Theorem~\ref{the:second} applies to tensor-product Galerkin approximations of the solution
to a general variational problem of the type~\eqref{wf} with
symmetric bilinear form $B$ that is continuous and elliptic on a Bochner-type space $V$.
\end{remark}

Recalling that $e^* = u_{X\mathcal{P}}^* - u_{X\mathcal{P}}$ and
putting together (\ref{eSeT}), (\ref{errest}) and (\ref{errest2}),
the following theorem gives two-sided bounds for the error reduction
$\norm{u_{X \mathcal{P}}^* - u_{X \mathcal{P}}}_B$ \rev{in} terms of the estimate $\eta$.

\begin{theorem} \label{the:third}
Let $u_{X \mathcal{P}} \in V_{X \mathcal{P}}$ be the Galerkin approximation satisfying \eqref{dwf},
and let $u_{X \mathcal{P}}^* \in V_{X \mathcal{P}}^*$ be the enhanced Galerkin approximation satisfying \eqref{dwfenhanced}.
Suppose that the norm equivalence \eqref{normequiv} holds.
Then the following estimates for the error reduction~hold:
\begin{align} \label{t2errered}
\frac{\lambda }{\sqrt{2}} \eta\leq \norm{u_{X \mathcal{P}}^* - u_{X \mathcal{P}}}_B \leq
\frac{\Lambda }{\sqrt{(1-\kappa_1)(1-\kappa_2^2)}}\eta,
\end{align}
where
$\lambda$ and $\Lambda$ are the constants in~\eqref{normequiv} and
$\kappa_1,\, \kappa_2 \in [0,1)$ are the constants in the strengthened Cauchy--Schwarz inequalities~\eqref{sCS1},~\eqref{sCS2}.

Furthermore, 
if $\kappa_1 = 0$ in \eqref{sCS1} (that is, $V_{X^*\mathcal{P}}$ and $V_{X\mathcal{Q}}$ are ${\widetilde{B}}$-orthogonal), then
\begin{align} \label{t2errred}
\lambda \eta  \leq \norm{u_{X \mathcal{P}}^* - u_{X \mathcal{P}}}_B \leq \frac{\Lambda }{\sqrt{1-\kappa_2^2}}\eta.
\end{align} 
\end{theorem}

\begin{remark}
Theorem~\ref{the:third} states that $\eta$ provides an estimate for the error reduction
$\norm{u_{X \mathcal{P}}^* - u_{X \mathcal{P}}}_B$.
We distinguish the following two important cases of enriching the approximation space~$V_{X \mathcal{P}}$:
\begin{itemize}
\item[(1)]
If only the finite element space is enriched, that is,
$V_{X \mathcal{P}}^* = V_{X^* \mathcal{P}} = V_{X \mathcal{P}} \oplus V_{Y \mathcal{P}}$, and 
$u_{X^*\mathcal{P}} \in V_{X^* \mathcal{P}}$ denotes the enhanced Galerkin solution, then
$\kappa_1 = 0$ and therefore $\eta = \norm{e_{Y \mathcal{P}}}_{\widetilde{B}}$ provides an effective estimate
for the error reduction $\norm{u_{X^* \mathcal{P}} - u_{X \mathcal{P}}}_B$,~i.e.,
\begin{align} \label{eq:err:red:eyp}
\lambda \norm{e_{Y \mathcal{P}}}_{\widetilde{B}}  \leq \norm{u_{X^* \mathcal{P}} - u_{X \mathcal{P}}}_B \leq \frac{\Lambda }{\sqrt{1-\kappa_2^2}} \norm{e_{Y \mathcal{P}}}_{\widetilde{B}}.
\end{align}

\item[(2)]
If only the polynomial space on $\Gamma$ is enriched, that is,
$V_{X \mathcal{P}}^* = V_{X \mathcal{P}^*} := V_{X \mathcal{P}} \oplus V_{X \mathcal{Q}}$, and 
$u_{X\mathcal{P}^*} \in V_{X \mathcal{P}^*}$ denotes the corresponding enhanced Galerkin solution, then
$\kappa_2 = 0$ and therefore $\eta = \norm{e_{X \mathcal{Q}}}_{\widetilde{B}}$ provides an effective estimate for the error reduction
$\norm{u_{X \mathcal{P}^*} - u_{X \mathcal{P}}}_B$,~i.e.,
\begin{align*}
\frac{\lambda }{\sqrt{2}} \norm{e_{X \mathcal{Q}}}_{\widetilde{B}} \leq
\norm{u_{X \mathcal{P}^*} - u_{X \mathcal{P}}}_B \leq \frac{\Lambda }{\sqrt{\rev{1-\kappa_1}}} \norm{e_{X \mathcal{Q}}}_{\widetilde{B}},
\end{align*}
when $\kappa_1 \neq 0$, and
\begin{align} \label{eq:err:red:exq:old}
\lambda \norm{e_{X \mathcal{Q}}}_{\widetilde{B}}  \leq
\norm{u_{X \mathcal{P}^*} - u_{X \mathcal{P}}}_B \leq \Lambda \norm{e_{X \mathcal{Q}}}_{\widetilde{B}},
\end{align} 
when $\kappa_1 = 0$.
\end{itemize}
Similar to Remark~\ref{rem:err:est}, we emphasize that Theorem~\ref{the:third} generalizes
the results of~\cite{2014-Bespalov-vol36,2016-Bespalov-vol38}, where
the error reduction estimates~\eqref{eq:err:red:eyp}, \eqref{eq:err:red:exq:old} have been proved
for the model problem~\eqref{paraPDE} with
the diffusion coefficient $T(\bm{x}, \bm{y})$ that has affine dependence on random parameters.

\end{remark}

\section{Galerkin approximations for the model problem with coefficient in the gPC expansion form} \label{sec:diffcoe}

While the results of section~\ref{sec:apprerr} hold for a general variational problem of type~\eqref{wf}
(see, e.g., Remark~\ref{rem:err:est}),
we now focus on the steady-state diffusion problem~\eqref{paraPDE}.
For this problem, we use the generalized polynomial chaos expansion
of the diffusion coefficient $T(\bm{x}, \bm{y})$ and specify main ingredients
of computing stochastic Galerkin approximations and the associated error estimators.
Here, and in the rest of the paper, we assume that $T(\bm{x}, \bm{y})$ depends on \emph{finite} number of parameters
$y_m$ ($m = 1,\dots,M$, $M \in {\mathbb{N}}$). As before, we suppose that $T(\bm{x}, \bm{y})$ satisfies
the boundedness assumption~\eqref{boundT}.
Then $T(\bm{x}, \bm{y}) \in W$ can be represented using the gPC expansion
as follows (see, e.g.,~\cite{2002-Xiu-p619} or~\cite[Theorem~3.6]{2012-Ernst-p317}):
\begin{align} \label{coeexpansion}
   T(\bm{x}, \bm{y}) = \sum_{\bm{\gamma}\in \mathbb{N}_0^M} t_{\bm{\gamma}}(\bm{x}) p_{\bm{\gamma}}(\bm{y}),
\end{align}
where the orthonormality of the polynomial basis $\left\{p_{\bm{\gamma}}\right\}_{\bm{\gamma}\in \mathbb{N}_0^M}$ gives
\begin{align} \label{coeexpansioncoe}
    t_{\bm{\gamma}}(\bm{x}) & = \langle T, p_{\bm{\gamma}} \rangle_{\pi} =
    \int_{\Gamma} T(\bm{x}, \bm{y}) p_{\bm{\gamma}}(\bm{y}) q(\bm{y}) \dif \bm{y}\quad
    \forall \bm{\gamma}\in \mathbb{N}_0^M.
\end{align}

\subsection{Discrete formulation revisited} \label{sec:discrete:problem}

Recalling that $X = \spn\{\phi_1, \phi_2, \dots, \phi_{n_X}\}$ and 
$ P_\mathcal{P} = \spn \left\{p_{\bm{\alpha}};\; \bm{\alpha} \in \mathcal{P} \subset \mathbb{N}_0^M\right\}$,
we can write any $u \in V_{X\mathcal{P}} = X \otimes  P_\mathcal{P}$ as
\begin{align} \label{eq:uxp:expand}
   u(\bm{x},\bm{y}) = \sum_{i = 1}^{n_{X}} \sum_{\bm{\alpha}\in\mathcal{P}} u_{i,\bm{\alpha}} \phi_i (\bm{x}) p_{\bm{\alpha}}(\bm{y}),\quad
   u_{i,\bm{\alpha}} \in {\mathbb{R}}.
\end{align}
We note that given multi-indices $\bm{\alpha},\,\bm{\beta},\,\bm{\gamma} \in \mathbb{N}_0^M$,
the orthogonality of the polynomial basis
(with respect to the inner product~$\langle\cdot, \cdot\rangle_{\pi}$)
yields the following property:
\begin{align} \label{eq:3prod:zero}
   \int_{\Gamma}q(\bm{y}) p_{\bm{\alpha}}(\bm{y}) p_{\bm{\beta}}(\bm{y}) p_{\bm{\gamma}}(\bm{y}) \dif \bm{y} =
   \prod_{m=1}^M \int_{\Gamma_m}q_m(y_m) p_{\alpha_m}^m (y_m) p_{\beta_m}^m (y_m) p_{\gamma_m}^m (y_m)\dif y_m = 0
\end{align}
if there exists $m \in \{1,2,\dots,M\}$ such that 
the sum of any two of $\alpha_m$, $\beta_m$ and $\gamma_m$ is less than the third one.
Therefore, with two \emph{finite} index sets $\mathcal{P},\, \mathcal{Q} \subset \mathbb{N}_0^M$,
we obtain by using~\eqref{coeexpansion} in the definition~\eqref{tbf} of the bilinear form~$B(\cdot,\cdot)$
\begin{align} \label{eq:B:represent}
   B(u,v) = &
   \sum_{\bm{\gamma}\in \mathbb{N}_0^M}
   \int_{\Gamma} q p_{\bm{\gamma}} 
   \int_D t_{\bm{\gamma}} \nabla u \cdot \nabla v \dif \bm{x} \dif \bm{y}
   \nonumber\\
   = &
   \sum_{\bm{\gamma}\in \mathcal{N}(\mathcal{P}, \mathcal{Q})}
   \int_{\Gamma} q p_{\bm{\gamma}}
   \int_D t_{\bm{\gamma}} \nabla u \cdot \nabla v\dif \bm{x} \dif \bm{y}
    \quad \forall u \in V_{X\mathcal{P}}, \ \ \forall v \in V_{X\mathcal{Q}},
\end{align}
where
\begin{align} \label{eq:mathcalN:def}
   \mathcal{N}(\mathcal{P}, \mathcal{Q}) \coloneqq
   \big\{\bm{\gamma} \in \mathbb{N}_0^M;\; &
   \exists\, \bm{\alpha} \in \mathcal{P},\, \exists\, \bm{\beta} \in \mathcal{Q}, \text{ such that }
   \nonumber \\
   & |\alpha_m -\beta_m| \leq \gamma_m \leq \alpha_m + \beta_m,\ \forall\, m = 1,\dots,M\big\}.
\end{align}
Thus, by using the Galerkin projection~\rev{\eqref{dwf}} onto the finite-dimensional subspace $V_{X\mathcal{P}}$,
the infinite sum in the expansion~\eqref{coeexpansion} of $T(\bm{x}, \bm{y})$
is effectively truncated to the finite sum over the
indices~$\bm{\gamma} \in \mathcal{N}(\mathcal{P}, \mathcal{P})$\footnote{Note that if $P_{\mathcal{P}}$ is a set of
complete polynomials of total degree $\le d$, then $P_{\mathcal{N}(\mathcal{P}, \mathcal{P})}$
is a set of  complete polynomials of total degree $\le 2d$.}.
In particular, using the representation~\eqref{eq:uxp:expand} for the Galerkin approximation $u_{X\mathcal{P}} \in V_{X\mathcal{P}}$
and setting $v = \phi_j p_{\bm{\beta}}$ in~\eqref{dwf}, we obtain for all $j = 1, \dots, n_X$ and~$\bm{\beta} \in \mathcal{P}$
\begin{align} \label{fullGalerkin}
   \sum_{\bm{\gamma}\in \mathcal{N}(\mathcal{P}, \mathcal{P})}
   \sum_{i=1}^{n_X} \sum_{\bm{\alpha} \in \mathcal{P}} 
   u_{i,\bm{\alpha}}
   \int_D   t_{\bm{\gamma}} \nabla  \phi_i  \cdot \nabla \phi_j   \dif \bm{x}
   \int_{\Gamma} q p_{\bm{\alpha}} p_{\bm{\beta}} p_{\bm{\gamma}}\dif \bm{y} =
   \int_D f\phi_j \dif \bm{x} \int_{\Gamma} q p_{\bm{\beta}}\dif \bm{y}.
\end{align}
Hence, the discrete formulation~\eqref{dwf} results in
\fx{the linear system $A \bm{u} = \bm{b}$
with the matrix~$A$ and the right-hand side vector $\bm{b}$ being defined as follows:
\begin{align*}
& A \coloneqq \sum_{\bm{\gamma}\in \mathcal{N}(\mathcal{P}, \mathcal{P})}
   G_{\bm{\gamma}} \otimes K_{\bm{\gamma}}, \quad \bm{b} \coloneqq \bm{g} \otimes \bm{f}, \nonumber \\
& [G_{\bm{\gamma}}]_{\iota(\bm{\alpha}) \iota(\bm{\beta})} \coloneqq
   \int_{\Gamma} q p_{\bm{\alpha}} p_{\bm{\beta}} p_{\bm{\gamma}} \dif \bm{y},
   \quad \iota(\bm{\alpha}), \iota(\bm{\beta}) = 1, \dots, \#\mathcal{P}, \nonumber \\
& [K_{\bm{\gamma}}]_{ij} \coloneqq \int_D  t_{\bm{\gamma}}   \nabla  \phi_i  \cdot \nabla \phi_j  \dif \bm{x}, \quad i,j = 1, \dots, n_X, \nonumber \\
& [\bm{g}]_{\iota(\bm{\beta})} \coloneqq \int_{\Gamma} q p_{\bm{\beta}}\dif \bm{y}, \quad [\bm{f}]_j \coloneqq \int_D f \phi_j \dif \bm{x},
\end{align*}
where $\iota: \mathcal{P} \to \{1, \dots, \#\mathcal{P}\}$ is a bijection.
Thus, the $[i+(\iota(\bm{\alpha})-1) n_X]$-th entry of the solution vector $\bm{u}$ is given by $u_{i,\bm{\alpha}}$.}

\subsection{Auxiliary bilinear forms} \label{subsec:abf}

An important ingredient of the error estimation strategy described in section~\ref{sec:apprerr}
is the auxiliary bilinear form~$\widetilde{B}(\cdot, \cdot)$.
In this subsection, we consider two choices of $\widetilde{B}(\cdot, \cdot)$,
which both exploit the gPC expansion~(\ref{coeexpansion}) of the diffusion coefficient~$T(\bm{x}, \bm{y})$.

The first auxiliary bilinear form employs the parameter-free part $t_{\bm{0}}(\bm{x})$
in the expansion~\eqref{coeexpansion} of $T(\bm{x}, \bm{y})$:
\begin{align} \label{B0}
   B_0(u, v) \coloneqq
   \int_{\Gamma}q(\bm{y}) \int_D t_{\bm{0}}(\bm{x})
   \nabla u(\bm{x}, \bm{y}) \cdot \nabla v(\bm{x}, \bm{y})\dif \bm{x} \dif \bm{y}.
\end{align}
The auxiliary bilinear form of this type has been used in the a posteriori error analysis of the sGFEM
for problem~\eqref{paraPDE} with the diffusion coefficient $T(\bm{x}, \bm{y})$ having affine dependence on $y_m$
(see, e.g.,~\cite{2014-Bespalov-vol36,2016-Bespalov-vol38,2018-Bespalov-p243}).

Since $\int_{\Gamma} q(\bm{y}) \dif \bm{y} = 1$ and
$T(\bm{x}, \bm{y})$ is bounded (see~(\ref{boundT})), we deduce from~\eqref{coeexpansioncoe} that
\begin{align} \label{boundt0}
   \alpha_{\min} \leq t_{\bm{0}}(\bm{x}) \leq \alpha_{\max} \quad \forall \bm{x} \in D.
\end{align}
Hence, the symmetric bilinear form $B_0(\cdot, \cdot)$ is continuous and elliptic on $V$.
Therefore, it defines an inner product in $V$ which induces the norm
$\norm{v}_{B_0} \coloneqq B_0(v,v)^{1/2}$ that is equivalent to $\norm{v}_V$.
Specifically, using~(\ref{boundt0}), we obtain
\begin{align} \label{B0ineq}
   \alpha_{\min}  \norm{v}_V^2  \leq \fx{\norm{v}_{B_0}^2}\leq \alpha_{\max} \norm{v}_V^2  \quad \forall v \in V.
\end{align}
Furthermore, using (\ref{B0ineq}) together with (\ref{Bineq}), we show that the norm equivalence in~\eqref{normequiv}
holds with $\widetilde{B} = B_0$, $\lambda = \sqrt{\frac{\alpha_{\min}}{\alpha_{\max}}}$ and
$\Lambda = \sqrt{\frac{\alpha_{\max}}{\alpha_{\min}}}$.

Turning now to the error estimators $e_{Y\mathcal{P}}$ and $e_{X\mathcal{Q}}$ that are defined in~\eqref{eYPestimator}
and~\eqref{eXQestimator} by employing the bilinear form $\widetilde{B} = B_0$,
we use the same arguments as in~\S\ref{sec:discrete:problem} (see~\eqref{eq:3prod:zero}--\eqref{eq:mathcalN:def})
to rewrite~\eqref{eYPestimator} and~\eqref{eXQestimator} as follows:
\begin{align}
   B_0(e_{Y\mathcal{P}}, v) & =
   F(v) - \int_{\Gamma} q \int_D
   \bigg( \sum_{\bm{\gamma} \in \mathcal{N}(\mathcal{P}, \mathcal{P})} t_{\bm{\gamma}} p_{\bm{\gamma}} \bigg)
   \nabla u_{X\mathcal{P}} \cdot \nabla v \dif \bm{x} \dif \bm{y} \quad \forall v \in V_{Y\mathcal{P}},
   \label{eYPB0} \\[5pt]
   B_0(e_{X\mathcal{Q}}, v) & =
   F(v) - \int_{\Gamma} q \int_D
   \bigg( \sum_{\bm{\gamma} \in \mathcal{N}(\mathcal{P}, \mathcal{Q})} t_{\bm{\gamma}} p_{\bm{\gamma}} \bigg)
   \nabla u_{X\mathcal{P}} \cdot \nabla v \dif \bm{x} \dif \bm{y} \quad \forall v \in V_{X\mathcal{Q}}.
   \label{eXQB0}
\end{align}
Furthermore,
the definition of the bilinear form $B_0(\cdot,\cdot)$ in~\eqref{B0} and
the orthogonality of the polynomial basis $\left\{p_{\bm{\gamma}}\right\}_{\bm{\gamma}\in \mathcal{Q}}$
(with respect to the inner product $\langle\cdot, \cdot\rangle_{\pi}$)
imply the $B_0$-orthogonality of the direct sum decomposition
$V_{X\mathcal{Q}} = \oplus_{\bm{\mu}\in \mathcal{Q}} X \otimes  P_{\{\bm{\mu}\}}$ (cf.~(\ref{orXQ})).
Therefore, by Corollary~\ref{col1},
the error estimator $e_{X \mathcal{Q}}$ and its norm $\norm{e_{X \mathcal{Q}}}_{B_0}$ can be decomposed
into the contributions associated with individual indices $\bm{\mu} \in \mathcal{Q}$,
see~\eqref{eXQdecomp} and~\eqref{eXQestimatori} with~$\widetilde{B} = B_0$.

The construction of the auxiliary bilinear form $\widetilde{B}(\cdot, \cdot)$ can be linked to designing a preconditioner for
the coefficient matrix associated with the bilinear form $B(\cdot, \cdot)$.
Indeed, the coefficient matrix associated with the auxiliary bilinear form $B_0(\cdot, \cdot)$ has been used in many works as
a preconditioner (called the \emph{mean-based} preconditioner) for linear systems resulting
from sGFEM formulations of parametric PDE problems (see, e.g., \cite{2000-Pellissetti-p607,2009-Powell-p350}).
Conversely, if there exists a good preconditioner for the coefficient matrix associated with bilinear form $B(\cdot, \cdot)$,
then one can try to design the auxiliary bilinear form by mimicking the structure of that preconditioner.
The above reasoning motivates our second choice of the auxiliary bilinear form $\widetilde{B}(\cdot, \cdot)$.
Specifically,
motivated by the Kronecker product structure of the preconditioner proposed in~\cite{2010-Ullmann-p923},
we construct the following bilinear~form:
\begin{align}
\label{B1}
B_1(u, v) \coloneqq  \sum_{\bm{\gamma}\in \mathbb{N}_0^M}\int_{\Gamma}q(\bm{y}) p_{\bm{\gamma}} (\bm{y}) \int_D C_{\bm{\gamma}} t_{\bm{0}}(\bm{x}) \nabla u(\bm{x}, \bm{y}) \cdot \nabla v(\bm{x}, \bm{y})\dif \bm{x} \dif \bm{y},
\end{align}
where $C_{\bm{\gamma}} \in {\mathbb{R}}$ are chosen to minimize the quantity
\fx{
$\mathcal{S} \coloneqq \big\|\sum_{\bm{\gamma}\in \mathbb{N}_0^M} C_{\bm{\gamma}}t_{\bm{0}} p_{\bm{\gamma}} - T\big\|_W^2$.
Using the expansion (\ref{coeexpansion}) of $T$, we rewrite $\mathcal{S}$ as follows:
\begin{align*}
   \mathcal{S} = \bigg\|{\sum_{\bm{\gamma}\in \mathbb{N}_0^M} (C_{\bm{\gamma}}t_{\bm{0}} - t_{\bm{\gamma}})p_{\bm{\gamma}}}\bigg\|_W^2 =
   \int_{\Gamma} q(\bm{y}) \int_D
   \bigg(\sum_{\bm{\gamma}\in \mathbb{N}_0^M}
   (C_{\bm{\gamma}}t_{\bm{0}}(\bm{x}) - t_{\bm{\gamma}}(\bm{x})) p_{\bm{\gamma}}(\bm{y})\bigg)^2 \dif \bm{x} \dif \bm{y}.
\end{align*}
Hence, the values of $C_{\bm{\gamma}}$ can be found from the following equation:
\begin{align*}
   \frac{\partial \mathcal{S}}{\partial C_{\bm{\gamma}}} &=
   \frac{\partial \left(\int_{\Gamma} q(\bm{y})
   \int_D \left(\sum_{\bm{\gamma}'\in \mathbb{N}_0^M}
   \big(C_{\bm{\gamma}'} t_{\bm{0}}(\bm{x}) - t_{\bm{\gamma}'}(\bm{x})\big) P_{\bm{\gamma}'}(\bm{y}) \right)^2
   \dif \bm{x} \dif \bm{y}\right)}{\partial C_{\bm{\gamma}}} \nonumber \\[3pt]
   &=
   2 \sum_{\bm{\gamma}'\in \mathbb{N}_0^M}
   \int_{\Gamma} q(\bm{y}) p_{\bm{\gamma}'}(\bm{y}) p_{\bm{\gamma}}(\bm{y}) \dif \bm{y}
   \int_D \big(C_{\bm{\gamma}'}t_{\bm{0}}(\bm{x}) - t_{\bm{\gamma}'}(\bm{x})\big) t_{\bm{0}}(\bm{x}) \dif \bm{x} \nonumber \\[3pt]
   &=
   2 \int_D \big(C_{\bm{\gamma}}t_{\bm{0}}(\bm{x}) - t_{\bm{\gamma}}(\bm{x})\big) t_{\bm{0}}(\bm{x}) \dif \bm{x} = 0\quad
   \forall\,\bm{\gamma} \in \mathbb{N}^M_0.
\end{align*}
}
As a result, we have
\begin{align} \label{B1c}
   C_{\bm{\gamma}} =
   \frac{\int_D t_{\bm{\gamma}}(\bm{x}) t_{\bm{0}}(\bm{x})  \dif \bm{x} }{\norm{t_{\bm{0}}}_{L^2(D)}^2}\quad
   \forall\,\bm{\gamma} \in \mathbb{N}^M_0
\end{align}
(note that with these values of $C_{\bm{\gamma}}$, one has
$\big\|\sum_{\bm{\gamma}\in \mathbb{N}_0^M} C_{\bm{\gamma}}t_{\bm{0}} p_{\bm{\gamma}}\|_W \le
  \|T\|_W < +\infty).
$

Substituting \eqref{B1c}
into \eqref{B1} and using (\ref{coeexpansion}),
we rewrite $B_1(u, v)$ as follows:
\begin{align} \label{B1rw}
   B_1(u, v) & =
   \fx{\sum_{\bm{\gamma}\in \mathbb{N}_0^M}\int_{\Gamma}q(\bm{y})p_{\bm{\gamma}}(\bm{y})
   \int_D \frac{\int_D t_{\bm{\gamma}}(\bm{x}') t_{\bm{0}}(\bm{x}')  \dif \bm{x}' }{\norm{t_{\bm{0}}}_{L^2(D)}^2}
   t_{\bm{0}}(\bm{x}) \nabla u(\bm{x}, \bm{y}) \cdot \nabla v(\bm{x}, \bm{y})\dif \bm{x} \dif \bm{y}}
   \nonumber\\[3pt]
   & =
   \frac{ \int_{\Gamma} q(\bm{y}) \int_D \int_D t_{\bm{0}}(\bm{x'}) t_{\bm{0}}(\bm{x}) T(\bm{x'}, \bm{y})
           \nabla u(\bm{x}, \bm{y}) \cdot \nabla v(\bm{x}, \bm{y})
           \dif \bm{x'}\dif \bm{x} \dif \bm{y}}
          {\norm{t_{\bm{0}}}_{L^2(D)}^2}.
\end{align}
Using this representation of $B_1(\cdot, \cdot)$ as well as 
the boundedness of $T(\bm{x}, \bm{y})$ and $t_{\bm{0}}(\bm{x})$ (see~\eqref{boundT} and~\eqref{boundt0}, resp.), we conclude
that $B_1(\cdot, \cdot)$
defines an inner product in $V$ which induces the norm $\norm{v}_{B_1} \coloneqq B_1(v, v)^{1/2}$ that is equivalent to $\norm{v}_{V}$. In particular, there holds
\begin{align} \label{B1ineq}
   \frac{\alpha_{\min}^3}{\alpha_{\max}^2} \norm{v}_V^2 \leq \fx{\norm{v}_{B_1}^2} \leq
   \frac{\alpha_{\max}^3}{\alpha_{\min}^2} \norm{v}_V^2 \quad \forall v \in V.
\end{align}
Furthermore, using (\ref{B1ineq}) together with (\ref{Bineq}) shows that the norm equivalence in~\eqref{normequiv}
holds with $\widetilde{B} = B_1$, $\lambda = \big(\frac{\alpha_{\min}}{\alpha_{\max}}\big)^{3/2}$ and
$\Lambda = \big(\frac{\alpha_{\max}}{\alpha_{\min}}\big)^{3/2}$.

If the bilinear form $B_1(\cdot, \cdot)$ is employed to define the error estimators
$e_{Y\mathcal{P}} \in V_{Y\mathcal{P}}$ and $e_{X\mathcal{Q}}\in V_{X\mathcal{Q}}$,
then the associated discrete formulations~(\ref{eYPestimator}) and~(\ref{eXQestimator})
can be rewritten as follows (here, we use the same arguments as in~\S\ref{sec:discrete:problem}):
\begin{align} \label{eYPB1}
   \int_{\Gamma} q \cdot
   \bigg( \sum_{\bm{\gamma} \in \mathcal{N}(\mathcal{P}, \mathcal{P})} &
            C_{\bm{\gamma}} p_{\bm{\gamma}}
   \bigg)
   \int_D
   t_{\bm{0}} \nabla e_{Y\mathcal{P}} \cdot \nabla v \dif \bm{x} \dif \bm{y}
   \nonumber \\[5pt]
   & \!\!= F(v) -
   \int_{\Gamma} q \int_D
   \bigg( \sum_{\bm{\gamma} \in \mathcal{N}(\mathcal{P}, \mathcal{P})}
   t_{\bm{\gamma}} p_{\bm{\gamma}} \bigg) \nabla u_{X\mathcal{P}} \cdot \nabla v \dif \bm{x} \dif \bm{y}
   \quad \forall v \in V_{Y\mathcal{P}},\\
   \label{eXQB1}
   \int_{\Gamma} q \cdot
   \bigg(\sum_{\bm{\gamma}\in \mathcal{N}(\mathcal{Q}, \mathcal{Q})} &
            C_{\bm{\gamma}}  p_{\bm{\gamma}}
   \bigg)
   \int_D
   t_{\bm{0}} \nabla e_{X\mathcal{Q}} \cdot \nabla v\dif \bm{x} \dif \bm{y}
   \nonumber \\
   & \!\!= F(v) -
   \int_{\Gamma}q \int_D
   \bigg(\sum_{\bm{\gamma}\in \mathcal{N}(\mathcal{P}, \mathcal{Q})}
   t_{\bm{\gamma}}p_{\bm{\gamma}}\bigg) \nabla u_{X\mathcal{P}} \cdot \nabla v \dif \bm{x} \dif \bm{y}
   \quad \forall v \in V_{X\mathcal{Q}}.
\end{align}
Comparing the left-hand sides in \eqref{eYPB1}, \eqref{eXQB1} with those in~\eqref{eYPB0},~\eqref{eXQB0},
respectively, it is easy to see that the computational cost associated with assembling linear systems for computing
the error estimators $e_{Y\mathcal{P}}$ and $e_{X\mathcal{Q}}$
will be significantly lower if the bilinear form $B_0$ is employed to define these estimators.

\subsection{Detail index set} \label{subsec:dis}

We now discuss the construction of the detail index set $\mathcal{Q}$ for computing the error estimator
$e_{X\mathcal{Q}}$ defined by~(\ref{eXQestimator}) in the case when the diffusion coefficient~$T(\bm{x},\bm{y})$
is given by its gPC expansion~(\ref{coeexpansion}).
Let $\mathcal{J} \subset {\mathbb{N}}_0^M$ denote the index set
such that all non-zero terms in expansion~(\ref{coeexpansion})
are indexed by $\bm{\gamma} \in \mathcal{J}$.
We will distinguish between two cases:
(i) $\mathcal{J}$ is a \emph{finite} index set; and
(ii) $\mathcal{J}$ is an \emph{infinite} (countable) set.

If the auxiliary bilinear form $\widetilde{B}$ satisfies (\ref{orXQ})
(which is the case when $\widetilde{B} = B_0$),
then by Corollary~\ref{col1}, the estimator $e_{X\mathcal{Q}}$
is the sum of individual estimators $e_{X\mathcal{Q}}^{(\bm{\mu})}$ ($\bm{\mu} \in \mathcal{Q}$) satisfying~(\ref{eXQestimatori}).
In this case, for a given~$\bm{\mu} \in \mathcal{Q}$, $e_{X\mathcal{Q}}^{(\bm{\mu})} = 0$  if and only if
the right-hand side \fx{of} (\ref{eXQestimatori}) is equal to zero for all $v \in X\otimes P_{\{\bm{\mu}\}}$,
which is equivalent to
$B(u_{X\mathcal{P}}, v) = 0$ for all $v \in X\otimes P_{\{\bm{\mu}\}}$
(note that $F(v) = 0$ for all $v \in X\otimes P_{\{\bm{\mu}\}}$, since
$\bm{0} \notin \mathcal{Q}$ and hence $\bm{\mu} \not= \bm{0}$).
Assume that $\mathcal{J}$ is a finite index set.
Then, recalling the definition of $\mathcal{N}(\cdot,\cdot)$ in~\eqref{eq:mathcalN:def}
and the orthogonality property~\eqref{eq:3prod:zero},
we conclude that
$e_{X\mathcal{Q}}^{(\bm{\mu})} = 0$
for any $\bm{\mu} \in \mathbb{N}_0^M \backslash \mathcal{N}(\mathcal{P}, \mathcal{J})$.
Therefore, for a finite index set $\mathcal{J}$ and an auxiliary bilinear form $\widetilde{B}$ satisfying~(\ref{orXQ}),
a natural choice of the detail index set is $\mathcal{Q} := \mathcal{N}(\mathcal{P}, \mathcal{J})\backslash\mathcal{P}$.

If $\widetilde{B}$ does not satisfy (\ref{orXQ}) (which is the case when $\widetilde{B} = B_1$)
or $\mathcal{J}$ is an infinite index set, then, in general,
we can only build the \emph{finite} detail index set~$\mathcal{Q}$ heuristically.


\section{Numerical experiments: error estimation} \label{sec:numer:errest}

The aim of this section is to test the error estimation strategy from~\S\ref{sec:apprerr}
for the model problem~(\ref{paraPDE}) with a
\fx{non-affine parametric representation of the diffusion coefficient.}
To~that end, we set $f(\bm{x}) = 1$ and $T(\bm{x}, \bm{y}) = \exp(a(\bm{x}, \bm{y}))$,
where $a(\bm{x}, \bm{y})$ is represented as~follows:
\begin{align} \label{kl}
   a(\bm{x}, \bm{y}) = a_0(\bm{x}) + \sum_{m = 1}^M a_m(\bm{x}) y_m,\quad
   \bm{x} \in D,\ \bm{y} \in \Gamma.
\end{align}
Here, we assume that $y_m$ are the images of \emph{independent and identically distributed} random variables
\fx{that follow} the same truncated Gaussian probability density function
\begin{align} \label{pdftG}
   q_m(y_m) = \frac{\exp(-y_m^2/2\sigma_0^2)}{\sigma_0\sqrt{2\pi}\erf(1/\sqrt{2}\sigma_0)},
\end{align}
where $\erf(\cdot)$ is the error function and $\sigma_0$ is a parameter of the truncated Gaussian distribution
measuring the standard deviation.

Note that for $T = \exp(a)$ and $a$ given by~\eqref{kl}, the gPC expansion~(\ref{coeexpansion})
has infinite number of non-zero terms;
the formulae for calculating the expansion coefficients $t_{\bm{\gamma}}$
in this case are given in Appendix~\ref{appA}.
The following two examples of decompositions of $a(\bm{x}, \bm{y})$ are considered in our~experiments.

\begin{example} \label{tp1}
Let $D = (-1, 1)^2$.
We assume that $a(\bm{x}, \bm{y})$ is represented by a truncated Karhunen--Lo\`{e}ve expansion
of a second-order random field with the mean $\mathbb{E}[a] = 1$ and the covariance function given by
\begin{align} \label{cov}
   \hbox{\rm Cov}[a](\bm{x}, \bm{x'}) = \sigma^2 \exp\left( -\frac{\abs{x_1 - x_1'}}{\ell_1}  -\frac{\abs{x_2 - x_2'}}{\ell_2} \right),
\end{align}
where $\sigma$ is the standard deviation and $\ell_1$, $\ell_2$ are correlation lengths
(we set $\ell_1 = \ell_2 = 1$).

Thus, in \eqref{kl}, we have: $a_0 = 1$ and $a_m(\bm{x}) = \sqrt{\lambda_m} \varphi_m(\bm{x})$ ($m = 1,\ldots,M$),
where $\{(\lambda_m, \varphi_m)\}_{m=1}^\infty$ are the eigenpairs of the integral operator
$\int_D \hbox{\rm Cov}[a](\bm{x}, \bm{x'}) \varphi(\bm{x'}) \dif \bm{x'}$
(see, e.g.~\cite[pp.~28--29]{gs91}).
\end{example}

\begin{example} \label{tp2}
Let $D = (0, 1)^2$, $a_0 = 1$ and choose the spatial coefficient functions $a_m(\bm{x})$ ($m=1,\ldots,M$) in~\eqref{kl}
as those introduced in~\cite[section 11]{2014-Eigel-vol270}:
\begin{align} \label{exeam}
   a_m (\bm{x}) =  \bar{\alpha} m^{-\bar{\sigma}} \cos(2\pi \bar{\beta}_1(m)x_1) \cos(2\pi\bar{\beta}_2(m)x_2), \quad \bm{x} = (x_1,x_2) \in D.
\end{align}
Here, $\rev{\bar{\sigma}} > 1$  characterizes the decay rate of the amplitudes $\bar{\alpha} m^{-\rev{\bar{\sigma}}}$
of these coefficients (we set $\rev{\bar{\sigma}} = 2$ in our experiments),
$\bar{\alpha} > 0$, and $\bar{\beta}_1,\,\bar{\beta}_2$ are defined as
\begin{align*}
\bar{\beta}_1(m) = m - \bar{k}(m)(\bar{k}(m)+1)/2 \text{ \ and \ } 
\bar{\beta}_2(m) = \bar{k}(m) - \bar{\beta}_1(m)
\end{align*}
with
$\bar{k}(m) = \lfloor -1/2 + \sqrt{1/4+2m}\rfloor$.
 \end{example}

All experiments in this section and in section~\ref{sec:numer:adapt} were performed using the open source MATLAB toolbox
S-IFISS~\cite{SIFISS}.
In our computations, we use the finite element space $X = X(h)$ of bilinear ($Q1$) 
approximations on uniform grids $\Box_h$ of square elements with edge length $h$.
In this case, the detail finite element space $Y = Y(h)$ is the span of the set of bilinear bubble functions
corresponding to edge midpoints and element centroids of the grid.
For the polynomial approximation on $\Gamma$, we first
construct a polynomial basis in $L^2_\pi(\Gamma)$ by tensorizing
univariate orthonormal polynomials generated by the probability density function~\eqref{pdftG}
(these polynomials are known in the literature as Rys polynomials, see, e.g.,~\cite[Example~1.11]{Gautschi2004});
then we employ the set $P_{M,d}$
of complete polynomials of degree $\le d$ in $M$ variables,
$P_{M,d} \coloneqq \spn \left\{p_{\bm{\alpha}};\; \bm{\alpha} \in \mathcal{P}_{M,d}\right\}$, where
\[
  \mathcal{P}_{M,d} \coloneqq \left\{\bm{\alpha} = (\alpha_1, \cdots, \alpha_M)\in \mathbb{N}_0^M;\;
  \alpha_1 + \ldots \alpha_M \leq d \right\}.
\]
Thus, given $h$, $M$ and $d$, we compute the Galerkin approximation $u_{X \mathcal{P}} \in X(h) \otimes P_{M,d}$
satisfying~\eqref{dwf}.

The spatial error estimator $e_{Y \mathcal{P}}$ satisfying (\ref{eYPestimator}) is computed approximately
by using a standard element residual technique (see, e.g., \cite{Ainsworth2000}).
Specifically, we solve the following local residual problems associated with~(\ref{eYPestimator}):
find $e_{Y\mathcal{P}}|_S \in Y(h)|_S \otimes P_{\mathcal{P}}$ satisfying
\begin{align}
\label{localeYP}
\widetilde{B}_S(e_{Y\mathcal{P}}|_S, v) = F_S(v) &+\int_{\Gamma} q(\bm{y}) \int_S \nabla\cdot \left(T(\bm{x}, \bm{y})\nabla u_{X\mathcal{P}}(\bm{x}, \bm{y})\right) v(\bm{x}, \bm{y}) \dif \bm{x} \dif \bm{y} \nonumber \\
& - \frac{1}{2}\int_{\Gamma} q(\bm{y}) \int_{\partial S \backslash \partial D} T(s, \bm{y}) \left\llbracket \frac{\partial u_{X\mathcal{P}}}{\partial n}\right\rrbracket  v(s, \bm{y}) \dif s \dif \bm{y}
\end{align}
for any $v \in Y(h)|_S \otimes P_{\mathcal{P}}$.
Here, $\widetilde{B}_S$ and $F_S$ denote the elementwise auxiliary bilinear form and linear functional, respectively; $Y(h)|_S$ is the restriction of $Y(h)$ to the element $S\in \Box_h$; and $\left\llbracket \frac{\partial u_{X\mathcal{P}}}{\partial n}\right\rrbracket$ denotes the flux jump in the approximate solution $u_{X\mathcal{P}}$ across interelement edges.
The parametric error estimator $e_{X \mathcal{Q}}$ is computed by solving~\eqref{eXQestimator}
(see also~\eqref{eXQdecomp}--\eqref{eXQestimatori} in the case $\widetilde{B} = B_0$).
Then two total error estimates are computed as follows (see~\eqref{estimate} with $\widetilde{B} = B_0$ and $\widetilde{B} = B_1$, resp.):
\begin{align} \label{estB01}
   \eta_0 =
   \bigg( \sum_{S \in \Box_h}\norm{e_{Y\mathcal{P}}|_S}_{B_{0,S}}^2 +
   \sum_{\bm{\mu} \in \mathcal{Q}} \norm{e^{(\bm{\mu})}_{X \mathcal{Q}}}_{B_0}^2\bigg)^{1/2},
   \ \
   \eta_1 =
   \bigg( \sum_{S \in \Box_h}\norm{e_{Y\mathcal{P}}|_S}_{B_{1,S}}^2 + \norm{e_{X \mathcal{Q}}}_{B_1}^2\bigg)^{1/2}.
\end{align} 
In the experiments below, we will examine the quality of the error estimates $\eta_0$ and $\eta_1$
by computing the corresponding effectivity indices
\begin{align} \label{thetaeff}
   \Theta_{i} \coloneqq  \frac{\eta_i}{\sqrt{\norm{u_{\text{ref}}}_B^2 - \norm{u_{X\mathcal{P}}}_B^2}},\quad
   i = 0,1,
\end{align}
where $u_{\text{ref}} \in X(h_{\text{ref}}) \otimes P_{M, d_{\text{ref}}}$ is
an accurate (reference) solution computed using biquadratic ($Q_2$) approximations on a uniform grid $\Box_{h_{\text{ref}}}$
with $h_{\text{ref}} < h$ and an enriched polynomial space $P_{M, d_{\text{ref}}}$ with $d_{\text{ref}} > d$.

In the experiments below, we set $\sigma_0 = 1$ in (\ref{pdftG}) and fix $M = 3$, $d = 2$.

\begin{table}
\centering
\begin{tabular}{ccccccccc}  
\toprule
&\multicolumn{2}{c}{$\sigma = 0.2$}&\multicolumn{2}{c}{$\sigma = 0.4$}&\multicolumn{2}{c}{$\sigma = 0.6$}&\multicolumn{2}{c}{$\sigma = 0.8$}\\
\cmidrule(r){2-9}
$h$ & $\Theta_0$ & $\Theta_1$& $\Theta_0$ & $\Theta_1$ & $\Theta_0$ & $\Theta_1$ & $\Theta_0$ & $\Theta_1$\\
\midrule
$2^{-1}$       & 1.3275 & 1.3238& 1.3331 &  1.3186 & 1.3434 & 1.3125 & 1.3599 & 1.3093\\
$2^{-2}$      & 1.1370 & 1.1331 & 1.1531 & 1.1382 & 1.1807 & 1.1496 & 1.2206 & 1.1714       \\
$2^{-3}$      & 1.0198 & 1.0162 & 1.0389 & 1.0254 & 1.0715 & 1.0439 &1.1179 &1.0754     \\
$2^{-4}$      & 0.9513 & 0.9480 & 0.9708 & 0.9586 & 1.0042 & 0.9797  &1.0515 & 1.0143    \\
$2^{-5}$      & 0.9134 & 0.9104 & 0.9329 & 0.9215 & 0.9668 & 0.9441  &1.0142 & 0.9793    \\
\bottomrule
\end{tabular}
\caption{The effectivity indices for Galerkin approximations $u_{X\mathcal{P}} \in X(h) \otimes P_{3, 2}$
for the model problem~\eqref{paraPDE} with $T(\bm{x}, \bm{y}) = \exp(a(\bm{x}, \bm{y}))$
and the decomposition of $a(\bm{x}, \bm{y})$ as in Example~\ref{tp1} with $\ell_1 = \ell_2 =1$.
The fixed detail index set $\mathcal{Q} = \mathcal{P}_{3, 6}  \backslash \mathcal{P}_{3, 2}$
is employed to compute the underlying error estimates.}
\label{tab:eff_ind_1_tp1}
\end{table}

In the first set of experiments, we consider two model problems described above
and vary the parameters that characterize the magnitude of the spatial coefficient functions $a_m$
in~\eqref{kl} (i.e., the parameters $\sigma$ and $\bar{\alpha}$ in Examples~\ref{tp1} and~\ref{tp2}, respectively).
Specifically, we choose $\sigma,\,\bar{\alpha} \in \{0.2,\,0.4,\,0.6,\,0.8\}$.
For each problem, we use spatial grids of decreasing mesh size $h = 2^{-j} \sqrt{|D|}$ ($j = 2,\ldots,6$)
to compute a sequence of Galerkin approximations $u_{X \mathcal{P}}$
and the corresponding error estimates $\eta_0,\,\eta_1$ defined in~\eqref{estB01}.
In particular, the parametric error estimators $e_{X \mathcal{Q}}$ are computed with the 
detail index set $\mathcal{Q} = \mathcal{P}_{3, 6}  \backslash \mathcal{P}_{3, 2}$.
For each computed error estimate, we calculate the effectivity index via~\eqref{thetaeff}.
Here, we use the reference solutions
$u_{\text{ref}} \in X(h_{\text{ref}}) \otimes P_{3,4}$,
where we choose $h_{\text{ref}} = 2^{-7}\sqrt{|D|}$.
The results of these computations are presented in Table~\ref{tab:eff_ind_1_tp1} (for the decomposition of $a(\bm{x},\bm{y})$ in Example~\ref{tp1})
and in Table~\ref{tab:eff_ind_1_tp2} (for the decomposition in~Example~\ref{tp2}).

From Tables~\ref{tab:eff_ind_1_tp1} and~\ref{tab:eff_ind_1_tp2} we find that both effectivity indices $\Theta_0$ and $\Theta_1$
are close to unity and decrease
as the spatial grid is refined or the corresponding coefficient parameter ($\sigma$ or $\bar\alpha$) decreases.
We also observe that $\Theta_0 > \Theta_1$ in each case, and the difference between $\Theta_0$ and $\Theta_1$ grows
as $\sigma$ and $\bar\alpha$~increase.

In the second set of experiments, we consider the same model problems as in the first set of experiments
but choose larger problem parameters, namely $\sigma,\, \bar\alpha \in \{1,\,3,\,5\}$.
In each case, we compute the Galerkin approximation $u_{X\mathcal{P}} \in X(h) \otimes P_{3, 2}$
with fixed $h = 2^{-5}\sqrt{|D|}$.
For each Galerkin approximation, two sequences of error estimates $\{\eta_0\}$ and $\{\eta_1\}$
are computed with different detail index sets;
specifically, we use $\mathcal{Q} = \mathcal{P}_{3, \bar{d}}  \backslash \mathcal{P}_{3, 2}$
with $\bar{d} \in \{3,4,\ldots,7\}$.
Then, the effectivity index
is calculated for each error estimate;
here, we again use the corresponding reference solutions
$u_{\text{ref}} \in X(h_{\text{ref}}) \otimes P_{3,4}$ with $h_{\text{ref}} = 2^{-7}\sqrt{|D|}$.
The effectivity indices are reported in Table~\ref{tab:eff_ind_2_tp1} (for the decomposition of $a(\bm{x},\bm{y})$ in Example~\ref{tp1})
and in Table~\ref{tab:eff_ind_2_tp2} (for the decomposition in~Example~\ref{tp2}).

From Tables~\ref{tab:eff_ind_2_tp1} and~\ref{tab:eff_ind_2_tp2} we 
again observe that $\Theta_0 > \Theta_1$ for each fixed $\sigma$ (resp., $\bar\alpha$) and for each detail index set.
Furthermore, for fixed $\sigma$ and $\bar\alpha$,
the effectivity indices in each sequence $\{\Theta_0\}$ and $\{\Theta_1\}$ approach their limiting values as the detail index set expands.
This convergence to limiting values is faster for smaller values of $\sigma$ and $\bar\alpha$.
For all $\sigma$ in Table~\ref{tab:eff_ind_2_tp1},
the limiting values of $\Theta_0$ stay close to unity,
whereas the limiting values of $\Theta_1$ decrease rapidly away from unity as $\sigma$ increases
(see the last two columns in Table~\ref{tab:eff_ind_2_tp1}).
This shows a robustness of the error estimate $\eta_0$
with respect to the `roughness' of the parametric coefficient in Example~\ref{tp1}.
This difference between the limiting values of $\Theta_0$ and $\Theta_1$ is less pronounced
for the parametric coefficient in Example~\ref{tp2} for given values of $\bar\alpha$
(see the last two columns in Table~\ref{tab:eff_ind_2_tp2}).
We can see, however, a faster decay of $\Theta_1$ as $\bar\alpha$ increases,
which indicates a deterioration of quality of the error estimate $\eta_1$ for larger values of~$\bar\alpha$.

\begin{table}
\centering
\begin{tabular}{ccccccccc}  
\toprule
&\multicolumn{2}{c}{$\bar{\alpha} = 0.2$}&\multicolumn{2}{c}{$\bar{\alpha} = 0.4$}&\multicolumn{2}{c}{$\bar{\alpha} = 0.6$}&\multicolumn{2}{c}{$\bar{\alpha} = 0.8$}\\
\cmidrule(r){2-9}
$h$ & $\Theta_0$ & $\Theta_1$& $\Theta_0$ & $\Theta_1$ & $\Theta_0$ & $\Theta_1$ & $\Theta_0$ & $\Theta_1$\\
\midrule
$2^{-2}$       & 1.3785 & 1.3783& 1.5197 &  1.5177 & 1.7139 & 1.7055 & 1.9284 & 1.9061\\
$2^{-3}$      & 1.1806 & 1.1805 & 1.3144 & 1.3126 & 1.5056 & 1.4975 & 1.7288 & 1.7066       \\
$2^{-4}$      & 1.0674 & 1.0673 & 1.2133 & 1.2113 & 1.4193 & 1.4108 &1.6575 &1.6344     \\
$2^{-5}$      & 1.0029 & 0.0028 & 1.1583 & 1.1563 & 1.3742 & 1.3655  &1.6177 & 1.5941    \\
$2^{-6}$      & 0.9678 & 0.9676 & 1.1289 & 1.1268 & 1.3486 & 1.3397  &1.5808 & 1.5571    \\
\bottomrule
\end{tabular}
\caption{The effectivity indices for Galerkin approximations $u_{X\mathcal{P}} \in X(h) \otimes P_{3, 2}$
for the model problem~\eqref{paraPDE} with $T(\bm{x}, \bm{y}) = \exp(a(\bm{x}, \bm{y}))$
and the decomposition of $a(\bm{x}, \bm{y})$ as in Example~\ref{tp2} with $\bar{\sigma} = 2$.
The fixed detail index set $\mathcal{Q} = \mathcal{P}_{3, 6}  \backslash \mathcal{P}_{3, 2}$
is employed to compute the underlying error estimates.}
\label{tab:eff_ind_1_tp2}
\end{table}

\begin{table}
\setlength{\tabcolsep}{5pt}
\centering
\begin{tabular}{ccccccccccccc}  
\toprule
&\multicolumn{2}{c}{$\mathcal{P}_{3, 3}  \backslash \mathcal{P}_{3, 2}$}&\multicolumn{2}{c}{$\mathcal{P}_{3, 4}  \backslash \mathcal{P}_{3, 2}$}&\multicolumn{2}{c}{$\mathcal{P}_{3, 5}  \backslash \mathcal{P}_{3, 2}$}&\multicolumn{2}{c}{$\mathcal{P}_{3, 6}  \backslash \mathcal{P}_{3, 2}$}&\multicolumn{2}{c}{$\mathcal{P}_{3, 7}  \backslash \mathcal{P}_{3, 2}$}\\
\cmidrule(r){2-11}
$\sigma$ & $\Theta_0$ & $\Theta_1$& $\Theta_0$ & $\Theta_1$ & $\Theta_0$ & $\Theta_1$& $\Theta_0$ & $\Theta_1$& $\Theta_0$ & $\Theta_1$\\
\midrule
1       &   1.1027 & 1.0591   & 1.1075 &  1.0601 & 1.1078& 1.0601  & 1.1078 &  1.0601 & 1.1078& 1.0601\\
3      & 0.7658 & 0.6889 & 1.0578 & 0.7796 & 1.1469 & 0.7849    & 1.1634 &  0.7862 & 1.1654 & 0.7863     \\
5      & 0.4398 & 0.3543 & 0.8163 & 0.2236 & 1.0517 & 0.5773    & 1.1512 &  0.5877 & 1.1813 & 0.5900   \\
\bottomrule
\end{tabular}
\caption{The effectivity indices for Galerkin approximations $u_{X\mathcal{P}} \in X(h) \otimes P_{3, 2}$
with $h = 2^{-4}$ for the model problem~\eqref{paraPDE} with $T(\bm{x}, \bm{y}) = \exp(a(\bm{x}, \bm{y}))$
and the decomposition of $a(\bm{x}, \bm{y})$ as in Example~\ref{tp1} with $\ell_1 = \ell_2 =1$.
The sequence of expanded index sets 
$\mathcal{Q} = \mathcal{P}_{3, \bar{d}}  \backslash \mathcal{P}_{3, 2}$
with $\bar{d} \in \{3,4,\ldots,7\}$ is employed to compute the underlying error estimates.}
\label{tab:eff_ind_2_tp1}
\end{table}

\begin{table}
\setlength{\tabcolsep}{5pt}
  \centering
\begin{tabular}{ccccccccccccc}  
\toprule
&\multicolumn{2}{c}{$\mathcal{P}_{3, 3}  \backslash \mathcal{P}_{3, 2}$}&\multicolumn{2}{c}{$\mathcal{P}_{3, 4}  \backslash \mathcal{P}_{3, 2}$}&\multicolumn{2}{c}{$\mathcal{P}_{3, 5}  \backslash \mathcal{P}_{3, 2}$}&\multicolumn{2}{c}{$\mathcal{P}_{3, 6}  \backslash \mathcal{P}_{3, 2}$}&\multicolumn{2}{c}{$\mathcal{P}_{3, 7}  \backslash \mathcal{P}_{3, 2}$}\\
\cmidrule(r){2-11}
$\bar{\alpha}$ & $\Theta_0$ & $\Theta_1$& $\Theta_0$ & $\Theta_1$ & $\Theta_0$ & $\Theta_1$& $\Theta_0$ & $\Theta_1$& $\Theta_0$ & $\Theta_1$\\
\midrule
1       &   1.8582 & 1.8090   & 1.8589 &  1.8097 & 1.8589& 1.8097  & 1.8589 &  1.8097 & 1.8589& 1.8097\\
3      & 1.6786 & 1.3352 & 1.7779 & 1.4150 & 1.7996 & 1.4276    & 1.8038 &  1.4331 & 1.8042 & 1.4341     \\
5      & 0.9080 & 0.7911 & 1.0825 & 0.8825 & 1.1786 & 0.9253    & 1.2218 &  0.9413 & 1.2353 & 0.9517   \\
\bottomrule
\end{tabular}
\caption{The effectivity indices for Galerkin approximations $u_{X\mathcal{P}} \in X(h) \otimes P_{3, 2}$
with $h = 2^{-5}$ for the model problem~\eqref{paraPDE} with $T(\bm{x}, \bm{y}) = \exp(a(\bm{x}, \bm{y}))$
and the decomposition of $a(\bm{x}, \bm{y})$ as in Example~\ref{tp2} with $\bar{\sigma} = 2$.
The sequence of expanded index sets 
$\mathcal{Q} = \mathcal{P}_{3, \bar{d}}  \backslash \mathcal{P}_{3, 2}$
with $\bar{d} \in \{3,4,\ldots,7\}$ is employed to compute the underlying error estimates.
}
\label{tab:eff_ind_2_tp2}
\end{table}

Based on the numerical results reported in this section, we conclude that
the bilinear form $\widetilde{B} = B_0$ is preferable to the bilinear form $\widetilde{B} = B_1$
for estimating the energy errors in sGFEM approximations for problems
with \fx{non-affine parametric representations of coefficients}.
Indeed, it follows from the numerical comparison of the associated effectivity indices that
the quality of the error estimate $\eta_0$ that employs $B_0$ is, in general, not worse than that
of the error estimate $\eta_1$ employing $B_1$.
Furthermore, as emphasized in~\S\ref{subsec:abf}, using the bilinear form $B_0$ is also preferable
from the computational cost point of view.
In addition to that, the $B_0$-orthogonality of the direct sum $\oplus_{\bm{\mu}\in \mathcal{Q}} X \otimes P_{\{\bm{\mu}\}}$
gives immediate access to individual parametric estimators $e^{(\bm{\mu})}_{X \mathcal{Q}}$ ($\bm{\mu} \in \mathcal{Q}$),
which is critical for building adaptive polynomial approximations on the parameter domain.
All this motivates the choice of the auxiliary bilinear form $\widetilde{B}$ and the associated energy error estimators
in the adaptive algorithm presented in the next section.

\section{Adaptive algorithm} \label{sec:adaptive}

In this section, we present an adaptive solution algorithm for the model problem~(\ref{paraPDE}).
\rev{Our focus here is on effective enrichment of the polynomial space in the parameter domain.}
We follow the ideas developed in~\cite{2016-Bespalov-vol38} but use D{\"o}rfler marking for enriching polynomial
approximations on $\Gamma$ and employ the error reduction estimates for \emph{marked} polynomial basis functions
in order to choose the refinement type (spatial \emph{vs.} parametric) at each iteration step (cf.~\cite[section~5]{2018-Bespalov-p243}).
\rev{The choice of D{\"o}rfler marking is motivated by the fact that it facilitates convergence analysis of adaptive algorithms
(cf.~\cite{doerfler, 2015-Eigel-p1367, 2019-Bespalov-p2359}); in particular, linear convergence of adaptive stochastic Galerkin
approximations is only proved in the case of D{\"o}rfler marking;
see~\cite[Theorem~7.2]{2015-Eigel-p1367} and~\cite[Theorem~8]{2019-Bespalov-p2359}.
}

Starting with a coarse grid of edge length $h_0$ and an initial index set $\mathcal{P}_0 \supseteq \mathcal{P}_{M,1}$,
the adaptive algorithm  generates a sequence of finite element spaces
\begin{align*}
X(h_0) \subseteq X(h_1) \subseteq X(h_2) \subseteq \cdots \subseteq X(h_K) \subset H_0^1(D),
\end{align*}
a sequence of polynomial spaces
\begin{align*}
P_{\mathcal{P}_0} \subseteq P_{\mathcal{P}_1} \subseteq P_{\mathcal{P}_2} \subseteq \cdots \subseteq
P_{\mathcal{P}_K} \subset L_{\pi}^2(\Gamma),
\end{align*}
and a sequence of Galerkin solutions $u^{(k)} \in V_{X\mathcal{P}}^k \coloneqq X(h_k) \otimes P_{\mathcal{P}_k}$.

At each iteration step $k$,  the Galerkin solution $u^{(k)}$ satisfying (\ref{dwf}) is computed by the subroutine \texttt{SOLVE} as follows:
\begin{align*}
u^{(k)} = \texttt{SOLVE}(T, f,  h_k, \mathcal{P}_k),
\end{align*}
where $T$ and $f$ are the problem data (see (\ref{paraPDE})).

At the error estimation step, we choose $\widetilde{B} = B_0$.
With this choice of $\widetilde{B}$, \fx{
we use (\ref{localeYP}) to compute the local (spatial) estimators $\{e_{Y\mathcal{P}}|_S\}_{S \in \Box_{h_k}}$
and employ~(\ref{eXQB0}) to compute the parametric estimator $e_{X\mathcal{Q}}$;
the latter gives access to 
individual parametric estimators $\big\{e_{X\mathcal{Q}}^{(\bm{\mu})}\big\}_{\bm{\mu}\in \mathcal{Q}_k}$ due to (\ref{eXQdecomp}).}
All estimators are computed by the subroutine \texttt{ESTIMATE}:
\begin{align*}
  \Big[ \left\{e_{Y\mathcal{P}}|_S;\; {S \in \Box_{h_k}}\right\},\;
  \big\{e_{X\mathcal{Q}}^{(\bm{\mu})};\; {\bm{\mu}\in \mathcal{Q}_k}\big\} \Big] =
  \texttt{ESTIMATE}(T, f,  h_k, \mathcal{P}_k, \mathcal{Q}_k,  u^{(k)}).
\end{align*}
Here, as discussed in section~\ref{subsec:dis}, the detail index set is built as follows:
if $\mathcal{J}$ is finite, then the natural choice of $\mathcal{Q}_k$ is
$\mathcal{Q}_k = \mathcal{N}(\mathcal{P}_k, \mathcal{J})\backslash\mathcal{P}_k$;
if $\mathcal{J}$ is infinite, then we build $\mathcal{Q}_k$ heuristically as
$\mathcal{Q}_k = \mathcal{N}(\mathcal{P}_k, \mathcal{N}(\mathcal{P}_k, \mathcal{P}_k))\backslash\mathcal{P}_k$.
Then we calculate the total error estimate $\eta^{(k)}$ via the first equation in~(\ref{estB01}).
\begin{algorithm}
\SetAlgoLined
\caption{Adaptive stochastic Galerkin finite element algorithm}
\label{algorithm}
\KwIn{data $T$, $f$; initial edge length $h_0$, initial index set $\mathcal{P}_0 \supseteq \mathcal{P}_{M,1}$;\\
\hspace{38pt}          marking threshold $\theta_{\mathcal{P}}$; tolerance $\epsilon$}
\KwOut{final Galerkin solution $u^{(K)}$, final error estimate $\eta^{(K)}$}
\For{$k = 0,1,2,\dots$}{
$u^{(k)} = \texttt{SOLVE}(T, f,  h_k, \mathcal{P}_k)$\;
$\Big[ \left\{e_{Y\mathcal{P}}|_S;\; {S \in \Box_{h_k}}\right\},\;
  \big\{e_{X\mathcal{Q}}^{(\bm{\mu})};\; {\bm{\mu}\in \mathcal{Q}_k}\big\} \Big] =
  \texttt{ESTIMATE}(T, f,  h_k, \mathcal{P}_k, \mathcal{Q}_k,  u^{(k)})$\; 
$\eta^{(k)} = \Big(\sum_{S \in \Box_{h_k}}\norm{e_{Y\mathcal{P}}|_S}_{B_{0,S}}^2 +
  \sum_{\bm{\mu}\in \mathcal{Q}_k} \big\|{e_{X\mathcal{Q}}^{(\bm{\mu})}}\big\|_{B_0}^2\Big)^{1/2}$\;
  \eIf{$\eta^{(k)} < \epsilon$}
{
$K \coloneqq k$;\ {\bf break}\;}
{
$\mathcal{M}_k =
  \texttt{MARK}\Big( \Big\{ \big\|{e_{X\mathcal{Q}}^{(\bm{\mu})}}\big\|_{B_0};\; \bm{\mu}\in \mathcal{Q}_k \Big\}, \theta_{\mathcal{P}}\Big)$\;
\eIf{$\sum_{S \in \Box_{h_k}} \norm{e_{Y\mathcal{P}}|_S}_{B_{0,S}}^2 \ge
         \sum_{\bm{\mu}\in \mathcal{M}_k} \big\|{e_{X\mathcal{Q}}^{(\bm{\mu})}}\big\|_{B_0}^2$}
 {
  $h_{k+1} = h_k/2$;\ \ $\mathcal{P}_{k+1} = \mathcal{P}_k$\;}
 {
  $h_{k+1} = h_k$;\ \ $\mathcal{P}_{k+1} = \mathcal{P}_k \cup  \mathcal{M}_k$\;}
}
}
\end{algorithm}

If the error estimate $\eta^{(k)}$ exceeds the prescribed tolerance $\epsilon$,
then an enriched finite-dimensional space $V_{X\mathcal{P}}^{k+1} \supset V_{X\mathcal{P}}^k$ must be constructed.
Before doing this, we identify those indices $\mu \in \mathcal{Q}_k$ that yield larger contributing estimators $e_{X\mathcal{Q}}^{(\mu)}$.
To that end, we employ the D{\"o}rfler marking strategy~\cite{doerfler}.
Specifically, we fix a threshold parameter $\theta_{\mathcal{P}} \in (0,1]$ and build a minimal subset
$\mathcal{M}_k \subseteq \mathcal{Q}_k$ 
such that
\begin{align} \label{eq:doerfler}
   \sum_{\bm{\mu} \in \mathcal{M}_k} \big\|{e^{(\bm{\mu})}_{X \mathcal{Q}}}\big\|_{B_0}^2 \geq
   \theta_{\mathcal{P}} \sum_{\bm{\mu} \in \mathcal{Q}_k} \big\|{e^{(\bm{\mu})}_{X \mathcal{Q}}}\big\|_{B_0}^2.
\end{align}
The marked index set is generated by the subroutine \texttt{MARK}:
\begin{align*}
   \mathcal{M}_k =
   \texttt{MARK}\Big( \Big\{ \big\|{e_{X\mathcal{Q}}^{(\bm{\mu})}}\big\|_{B_0};\; \bm{\mu}\in \mathcal{Q}_k \Big\}, \theta_{\mathcal{P}}\Big).
\end{align*}

\begin{table}
\setlength{\tabcolsep}{2pt}
  \centering
  \footnotesize
\begin{tabular}{ccccc}  
\toprule
$\sigma$ & 0.4 & 0.6 & 0.8 & 1 \\
\midrule
$t$, sec & 1.8181e+02 & 5.4331e+02 & 4.1346e+03 & 5.7141e+03 \\
$K$  & 4 & 5 & 6 & 6 \\
$\eta^{(K)}$  & 1.8628e-02 & 1.6762e-02 & 1.1463e-02 & 1.4294e-02 \\
\midrule
$h_K$ & $2^{-4}$ & $2^{-5}$  & $2^{-5}$ & $2^{-5}$ \\
$\# \mathcal{N}(\mathcal{P}_K, \mathcal{Q}_K)$ & 125 & 125 & 999 & 1,339 \\
$N_K$ & 6,534 & 25,350 & 71,825 & 80,275\\
\midrule
$\mathcal{P}$ & 
\begin{tabular}[t]{cc}
$k = 0$ & (0 0 0 0 0)\\
& (0 0 0 0 1) \\
& (0 0 0 1 0) \\
& (0 0 1 0 0) \\
& (0 1 0 0 0) \\
& (1 0 0 0 0)
\end{tabular}
 & 
 \begin{tabular}[t]{cc}
$k = 0$ & (0 0 0 0 0)\\
& (0 0 0 0 1) \\
& (0 0 0 1 0) \\
& (0 0 1 0 0) \\
& (0 1 0 0 0) \\
& (1 0 0 0 0)
\end{tabular}
&
\begin{tabular}[t]{cc}
$k = 0$ & (0 0 0 0 0)\\
& (0 0 0 0 1) \\
& (0 0 0 1 0) \\
& (0 0 1 0 0) \\
& (0 1 0 0 0) \\
& (1 0 0 0 0) \\
\midrule
$k = 5$ & (2 0 0 0 0) \\
& (1 1 0 0 0) \\
& (1 0 1 0 0) \\
& (1 0 0 1 0) \\
& (1 0 0 0 1) \\
& (0 1 1 0 0) \\
& (0 1 0 1 0) \\
& (0 0 1 0 1) \\
& (0 2 0 0 0) \\
& (0 0 2 0 0) \\
& (0 1 0 0 1)
\end{tabular}
&
\begin{tabular}[t]{cc}
$k = 0$ & (0 0 0 0 0)\\
& (0 0 0 0 1) \\
& (0 0 0 1 0) \\
& (0 0 1 0 0) \\
& (0 1 0 0 0) \\
& (1 0 0 0 0) \\
\midrule
$k = 5$ & (2 0 0 0 0) \\
& (1 1 0 0 0) \\
& (1 0 1 0 0) \\
& (1 0 0 1 0) \\
& (1 0 0 0 1) \\
& (0 1 1 0 0) \\
& (0 1 0 1 0) \\
& (0 0 1 0 1) \\
& (2 0 1 0 0) \\
& (2 1 0  0 0) \\
& (0 2 0 0 0)\\
& (0 0 2 0 0) \\
& (1 1 1 0 0)
\end{tabular}
\\
\bottomrule
\end{tabular}
\caption{The results of running Algorithm~\ref{algorithm} for the model problem~\eqref{paraPDE} with
$T(\bm{x}, \bm{y}) = \exp(a(\bm{x}, \bm{y}))$
and the decomposition of $a(\bm{x}, \bm{y})$ as in Example~\ref{tp1} with $\ell_1 = \ell_2 =1$.
}
\label{tab:adapt_1}
\end{table}
In order to construct the enriched approximation space, we either enrich the finite element space by uniformly refining the mesh
(in this case, we define $V_{X\mathcal{P}}^{k+1,1} \coloneqq X(h_{k+1}) \otimes P_{\mathcal{P}_k}$ with $h_{k+1} = h_k/2$), or
enrich the polynomial space by including the (marked) indices from $\mathcal{M}_k \subseteq \mathcal{Q}_k$
(i.e., we set $V_{X\mathcal{P}}^{k+1,2} \coloneqq X(h_k) \otimes P_{\mathcal{P}_{k+1}}$
with $\mathcal{P}_{k+1} = \mathcal{P}_k \cup \mathcal{M}_k$).
Let $u^{(k+1,l)} \in V_{X\mathcal{P}}^{k+1,l}$ ($l = 1,2$) denote the corresponding enhanced Galerkin approximations
(note that none of these approximations is computed at this stage).
In order to determine the refinement type (spatial or parametric), we recall that Theorem~\ref{the:third} implies that
$\left(\sum_{S \in \Box_h}\norm{e_{Y\mathcal{P}}|_S}_{B_{0,S}}^2 \right)^{1/2}$ and
$\Big(\sum_{\bm{\mu}\in \mathcal{M}_k} \big\|{e_{X\mathcal{Q}}^{(\bm{\mu})}}\big\|_{B_0}^2\Big)^{1/2}$
provide  effective estimates for the error reductions
$\big\|{u^{(k+1,1)} - u^{(k)}}\big\|_B$ and $\big\|{u^{(k+1,2)} - u^{(k)}}\big\|_B$, respectively.
Therefore, if $\left(\sum_{S \in \Box_h}\norm{e_{Y\mathcal{P}}|_S}_{B_{0,S}}^2 \right)^{1/2} $ is greater than or equal to
$\big(\sum_{\bm{\mu}\in \mathcal{M}_k} \big\|{e_{X\mathcal{Q}}^{(\bm{\mu})}}\big\|_{B_0}^2\big)^{1/2}$,
we define $V_{X\mathcal{P}}^{k+1} := V_{X\mathcal{P}}^{k+1,1} $, leading to spatial refinement;
otherwise, we set $V_{X\mathcal{P}}^{k+1} := V_{X\mathcal{P}}^{k+1,2} $, leading to parametric refinement.
Then a more accurate Galerkin solution $u^{(k+1)} \in V_{X\mathcal{P}}^{k+1}$ is computed.
The process is then repeated until the tolerance is met.

The complete adaptive algorithm is listed in Algorithm~\ref{algorithm}.

\section{Numerical experiments: adaptivity} \label{sec:numer:adapt}

\begin{figure}
\centering
\begin{overpic}[width=0.495\textwidth]{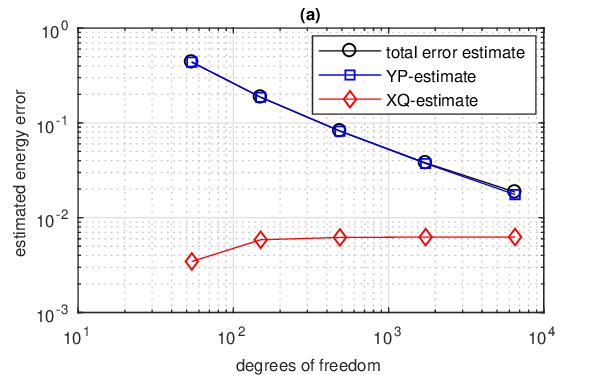}
\end{overpic}
\begin{overpic}[width=0.495\textwidth]{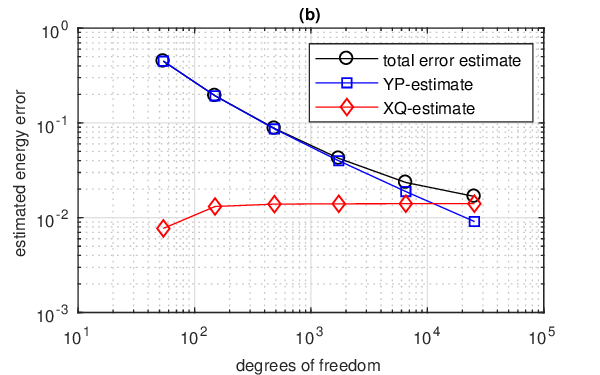}
\end{overpic}
\begin{overpic}[width=0.495\textwidth]{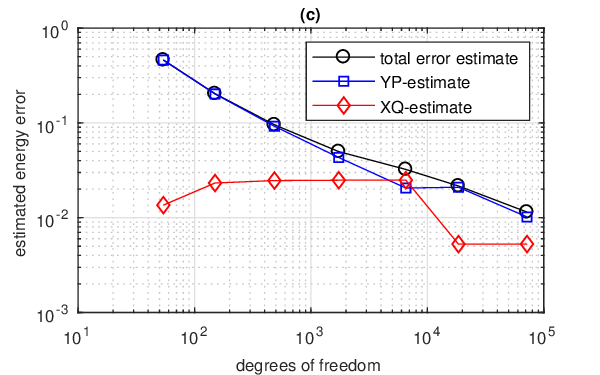}
\end{overpic}
\begin{overpic}[width=0.495\textwidth]{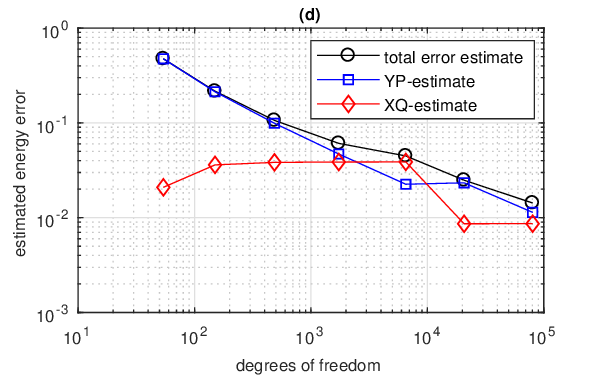}
\end{overpic}
\caption{Energy error estimates at each step of the adaptive algorithm for the model problem~\eqref{paraPDE} with
$T(\bm{x}, \bm{y}) = \exp(a(\bm{x}, \bm{y}))$
and the decomposition of $a(\bm{x}, \bm{y})$ as in Example~\ref{tp1} with $\ell_1 = \ell_2 =1$:
(a) $\sigma = 0.4$; (b) $\sigma = 0.6$; (c) $\sigma = 0.8$; (d) $\sigma = 1$.}
\label{fig:err_est_1}
\end{figure}

In this section, we test the performance of Algorithm~\ref{algorithm} 
for the model problem~(\ref{paraPDE}) with
\fx{non-affine parametric representations of the diffusion coefficient}.
As in section~\ref{sec:numer:errest}, numerical results are presented for bilinear (Q1) spatial approximations
on uniform grids $\Box_h$ of square elements with edge length $h$.
In all experiments, 
we set the marking parameter $\theta_{\mathcal{P}} = 0.9$ in~\eqref{eq:doerfler}
and run the adaptive algorithm with the stopping tolerance $\epsilon = 2 \times 10^{-2}$.

In our first set of experiments in this section, we consider the model problem~(\ref{paraPDE}) on the domain
$D = (-1,1)^2$ and we set $f(\bm{x}) = 1$, $T(\bm{x}, \bm{y}) = \exp(a(\bm{x}, \bm{y}))$,
where $a(\bm{x}, \bm{y})$ is represented as in~\eqref{kl} by using
the truncated Karhunen--Lo\`{e}ve expansion in Example~\ref{tp1}.
As in section~\ref{sec:numer:errest}, we assume that
$y_m$ are the images of independent and identically distributed random variables
\fx{that follow} the truncated Gaussian probability density function in~\eqref{pdftG} with  $\sigma_0 = 1$.
We fix $M = 5$ in~\eqref{kl} and for each $\sigma \in \{0.4,\, 0.6,\, 0.8,\, 1\}$ in~(\ref{cov}) we run the adaptive algorithm.
The results of these computations are presented in Table~\ref{tab:adapt_1} and Figures~\ref{fig:err_est_1} and~\ref{fig:eff_ind_1}.

\begin{figure}
\centering
\begin{overpic}[width=0.495\textwidth]{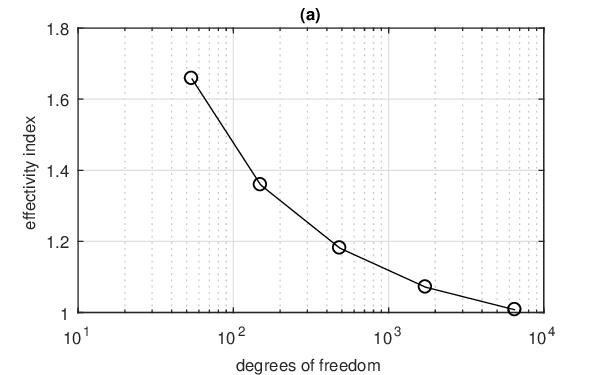}
\end{overpic}
\begin{overpic}[width=0.495\textwidth]{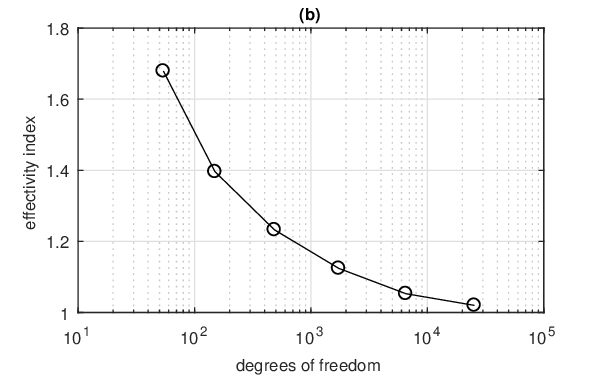}
\end{overpic}
\begin{overpic}[width=0.495\textwidth]{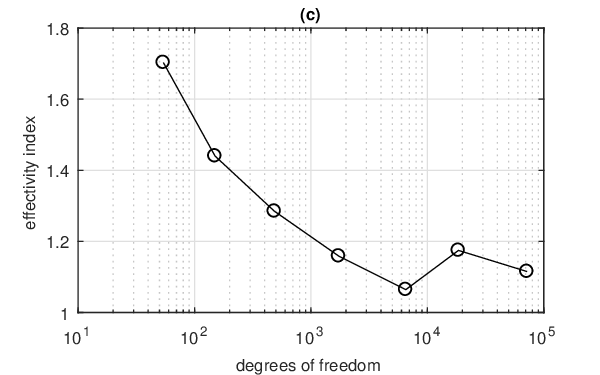}
\end{overpic}
\begin{overpic}[width=0.495\textwidth]{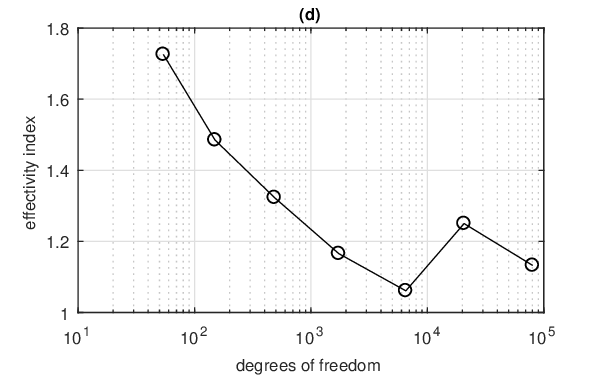}
\end{overpic}
\caption{The effectivity indices for the Galerkin solutions at each iteration of the adaptive algorithm
for the model problem~\eqref{paraPDE} with $T(\bm{x}, \bm{y}) = \exp(a(\bm{x}, \bm{y}))$
and the decomposition of $a(\bm{x}, \bm{y})$ as in Example~\ref{tp1} with $\ell_1 = \ell_2 =1$:
(a) $\sigma = 0.4$; (b) $\sigma = 0.6$; (c) $\sigma = 0.8$; (d)~$\sigma = 1$.}
\label{fig:eff_ind_1}
\end{figure}

In Table~\ref{tab:adapt_1}, for each computation we report the overall computational time $t$ (in seconds),
the number of iterations $K$ needed to reach the prescribed tolerance,
the final error estimate $\eta^{(K)}$, the edge length $h_K$ of the final mesh,
the cardinality of the index set $\mathcal{N}(\mathcal{P}_K, \mathcal{Q}_K)$ that is used in
calculating the estimator $e_{X\mathcal{Q}}$ (see~\eqref{eXQB0}),
the final number of degrees of freedom $N_K := \text{dim}(V_{X\mathcal{P}}^K)$,
and the evolution of the index set~$\mathcal{P}$.

From Table~\ref{tab:adapt_1}, we find that in the experiments with larger values of $\sigma$,
the tolerance is met by the final Galerkin solution calculated on a more refined spatial grid $\Box_{h_K}$
and with a larger index set $\mathcal{P}_K$.
This leads to significant increase in computational times and is due to a dramatic expansion
of the index set $\mathcal{N}(\mathcal{P}_K, \mathcal{Q}_K)$ as $\sigma$ increases.
For example, as $\sigma$ increases from $0.6$ to $0.8$, the cardinality of $\mathcal{N}(\mathcal{P}_K, \mathcal{Q}_K)$
increases by approximately a factor of 8, while the cardinality of $\mathcal{P}_K$ only increases by approximately a factor of~3.
Greater cardinality of $\mathcal{N}(\mathcal{P}_K, \mathcal{Q}_K)$ means longer computational time
for finding $e_{X\mathcal{Q}}$ via~\eqref{eXQB0}, taking a significant share of the overall computational time.

In Figure~\ref{fig:err_est_1}, we plot the error estimates
$\eta$, $\norm{e_{Y \mathcal{P}}}_{B_0}$ and $\norm{e_{X \mathcal{Q}}}_{B_0}$ at each iteration of the adaptive loop.
In Figure~\ref{fig:err_est_1}(a) (i.e., for $\sigma = 0.4$),
we observe that $\norm{e_{Y \mathcal{P}}}_{B_0}$ is greater than $\norm{e_{X \mathcal{Q}}}_{B_0}$
throughout the computation, that is why no parametric refinement is performed in this case before the tolerance is met.
In Figure~\ref{fig:err_est_1}(b) ($\sigma = 0.6$), we find that $\norm{e_{Y \mathcal{P}}}_{B_0}$ is only smaller than $\norm{e_{X \mathcal{Q}}}_{B_0}$
at the final iteration, when the total error estimate is below the tolerance; thus no parametric refinement is performed in this case either.
In the ex\-pe\-ri\-ments with $\sigma = 0.8$ and $\sigma = 1$, one parametric refinement is needed before the tolerance is met
(see Figures~\ref{fig:err_est_1}(c) and~\ref{fig:err_est_1}(d)).
Note that more indices were activated in the case of~$\sigma = 1$.

\begin{table}
\setlength{\tabcolsep}{2pt}
  \centering
  \footnotesize
\begin{tabular}{ccccc}  
\toprule
$\bar{\alpha}$ & 0.4 & 0.6 & 0.8 & 1 \\
\midrule
$t$, sec & 1.5620e+01 & 1.8369e+01 & 3.2806e+01 & 1.1202e+02 \\
$K$  & 4 & 5 & 6 & 8 \\
$\eta^{(K)}$  & 1.4633e-02 & 1.8236e-02 & 1.6919e-02 & 1.8534e-02 \\
\midrule
$h_K$ & $2^{-5}$ & $2^{-5}$  & $2^{-6}$ & $2^{-7}$ \\
$\# (\mathcal{N}(\mathcal{P}_K, \mathcal{Q}_K) \cap \mathcal{P}_{5,2})$ & 20 & 20 & 20 & 20 \\
$N_K$ & 6,534 & 9,801 & 38,025 & 249,615\\
\midrule
$\mathcal{P}$ & 
\begin{tabular}[t]{cc}
$k = 0$ & (0 0 0 0 0)\\
& (0 0 0 0 1) \\
& (0 0 0 1 0) \\
& (0 0 1 0 0) \\
& (0 1 0 0 0) \\
& (1 0 0 0 0)
\end{tabular}
 & 
 \begin{tabular}[t]{cc}
$k = 0$ & (0 0 0 0 0)\\
& (0 0 0 0 1) \\
& (0 0 0 1 0) \\
& (0 0 1 0 0) \\
& (0 1 0 0 0) \\
& (1 0 0 0 0) \\ 
\midrule
$k = 5$ &(2 0 0 0 0) \\
& (1 1 0 0 0) \\
& (3 0 0 0 0)
\end{tabular}
&
\begin{tabular}[t]{cc}
$k = 0$ & (0 0 0 0 0)\\
& (0 0 0 0 1) \\
& (0 0 0 1 0) \\
& (0 0 1 0 0) \\
& (0 1 0 0 0) \\
& (1 0 0 0 0) \\
\midrule
$k = 5$ &(2 0 0 0 0) \\
& (1 1 0 0 0) \\
& (3 0 0 0 0)
\end{tabular}
&
\begin{tabular}[t]{cc}
$k = 0$ & (0 0 0 0 0)\\
& (0 0 0 0 1) \\
& (0 0 0 1 0) \\
& (0 0 1 0 0) \\
& (0 1 0 0 0) \\
& (1 0 0 0 0) \\
\midrule
$k = 5$ &(2 0 0 0 0) \\
& (1 1 0 0 0) \\
& (3 0 0 0 0) \\
\midrule
$k = 7$ & (2 1 0 0 0) \\
& (4 0 0 0 0) \\
& (1 0 1 0 0) \\
& (5 0 0 0 0) \\
& (3 1 0 0 0) \\
& (1 0 0 1 0)
\end{tabular}
\\
\bottomrule
\end{tabular}
\caption{The results of running Algorithm~\ref{algorithm} for the model problem~\eqref{paraPDE} with
$T(\bm{x}, \bm{y}) = a^2(\bm{x}, \bm{y})$
and the decomposition of $a(\bm{x}, \bm{y})$ as in Example~\ref{tp2} with $\bar{\sigma} = 2$.}
\label{tab:adapt_2}
\end{table}

In Figure~\ref{fig:eff_ind_1}, we plot the effectivity indices computed via~(\ref{thetaeff}) with $i=0$ at each iteration of the algorithm.
Here, the reference solution $u_{\text{ref}}$ in each experiment is computed using biquadratic (Q2) spatial approximations
on a fine grid $\Box_{h_{\text{ref}}}$ with $h_{\text{ref}} = h_K/2$ and employing the polynomial space \fx{$P_{M,d_K +1}$}
with \fx{$d_K$} being the highest \fx{(total)} degree of the polynomials in $P_{\mathcal{P}_K}$.
We can see that for all experiments the effectivity indices are within the interval $(1, 2)$ throughout all iterations.

In the final set of experiments, we consider the model problem~(\ref{paraPDE}) on the domain
$D = (0,1)^2$ and we set $f(\bm{x}) = 1$, $T(\bm{x}, \bm{y}) = a^2(\bm{x}, \bm{y})$,
where $a(\bm{x}, \bm{y})$ is represented as in~\eqref{kl} with
the coefficient functions $a_m$ ($m = 0,\ldots,M$) chosen as in Example~\ref{tp2}.
We again assume that $y_m$ are the images of independent and identically distributed random variables
\fx{that follow} the truncated Gaussian probability density function in~\eqref{pdftG} with $\sigma_0 = 1$.
Note that for $T = a^2$ and $a$ given by~\eqref{kl}, the gPC expansion~(\ref{coeexpansion})
has a finite number of non-zero terms;
the formulae for calculating the expansion coefficients $t_{\bm{\gamma}}$
in this case are given in Appendix~\ref{appA}.
We fix $M = 5$ in~\eqref{kl} and
for each $\bar{\alpha} \in \{0.4,\,0.6,\,0.8,\,1\}$ in (\ref{exeam}) we run the adaptive algorithm.
The results of computations are presented in Table~\ref{tab:adapt_2} and Figures~\ref{fig:err_est_2} and~\ref{fig:eff_ind_2}.

From Table~\ref{tab:adapt_2} and Figure~\ref{fig:err_est_2}, we find that no parametric refinement
is performed in the experiment with $\bar{\alpha} = 0.4$;
one parametric refinement is performed in the experiments with $\bar{\alpha} = 0.6$ and $\bar{\alpha} = 0.8$;
and two parametric refinements are performed in the experiment with $\bar{\alpha} = 1$.
Since for the model problem in this set of experiments, the gPC expansion~\eqref{coeexpansion} of the diffusion coefficient $T = a^2$
reduces to a finite sum over the index set $\mathcal{J} \subset \mathcal{P}_{5,2}$ (see Appendix~\ref{appA}),
the sum in~(\ref{eXQB0}) is over the set
$ \mathcal{N}(\mathcal{P}_K, \mathcal{Q}_K) \cap \mathcal{P}_{5,2}$.
We observe from Table~\ref{tab:adapt_2} that
$\#(\mathcal{N}(\mathcal{P}_K, \mathcal{Q}_K) \cap \mathcal{P}_{5,2})$
does not change throughout this set of experiments.
This partly explains the reason why the overall computational times for larger values of coefficient parameter
(i.e., the parameter $\bar\alpha$ in this set of experiments)
do not increase as significantly as they do in the first set of experiments in this~section.

In Figure~\ref{fig:eff_ind_2}, we plot the effectivity indices for the error estimate at each iteration of the algorithm
(here, the reference Galerkin solution is computed similarly to other experiments).
We can see that for all experiments in this set, the effectivity indices are within the interval $(0.5, 2.5)$ throughout all iterations.

\begin{figure}
\centering
\begin{overpic}[width=0.495\textwidth]{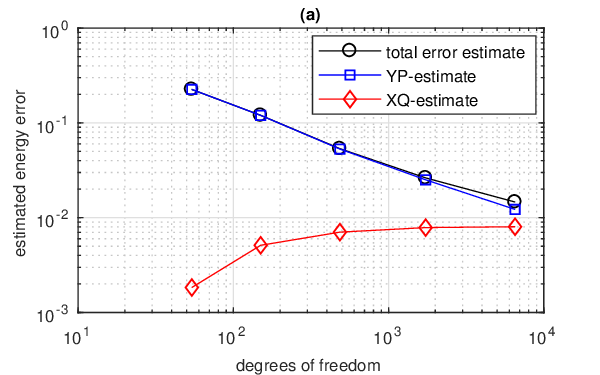}
\end{overpic}
\begin{overpic}[width=0.495\textwidth]{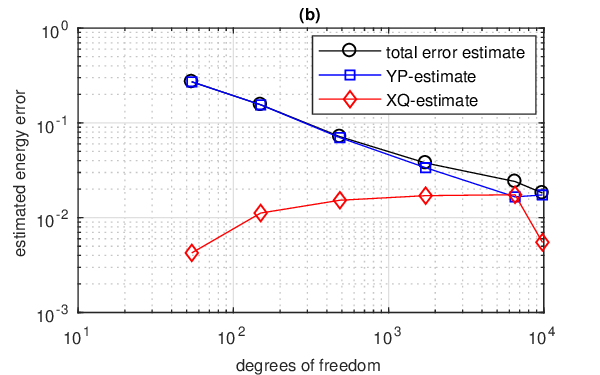}
\end{overpic}
\begin{overpic}[width=0.495\textwidth]{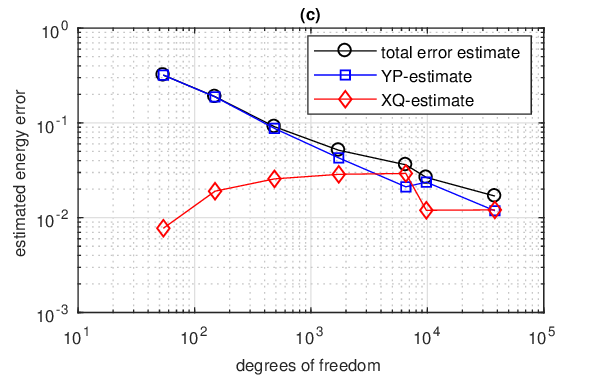}
\end{overpic}
\begin{overpic}[width=0.495\textwidth]{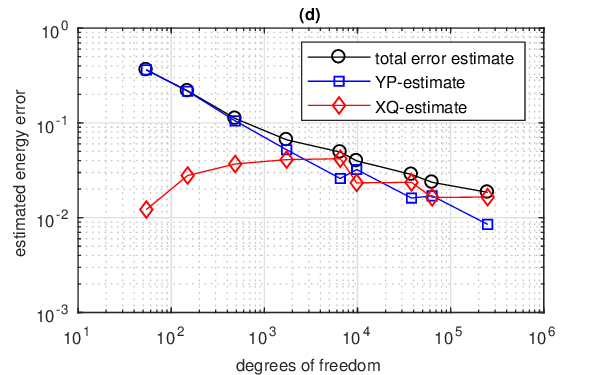}
\end{overpic}
\caption{Energy error estimates at each step of the adaptive algorithm for the model problem~\eqref{paraPDE} with
$T(\bm{x}, \bm{y}) = a^2(\bm{x}, \bm{y})$ and the decomposition of $a(\bm{x}, \bm{y})$ as in Example~\ref{tp2} with $\bar{\sigma} = 2$:
(a) $\bar{\alpha} = 0.4$; (b) $\bar{\alpha} = 0.6$; (c) $\bar{\alpha} = 0.8$; (d) $\bar{\alpha} = 1$.}
\label{fig:err_est_2}
\end{figure}

\begin{figure}
\centering
\begin{overpic}[width=0.495\textwidth]{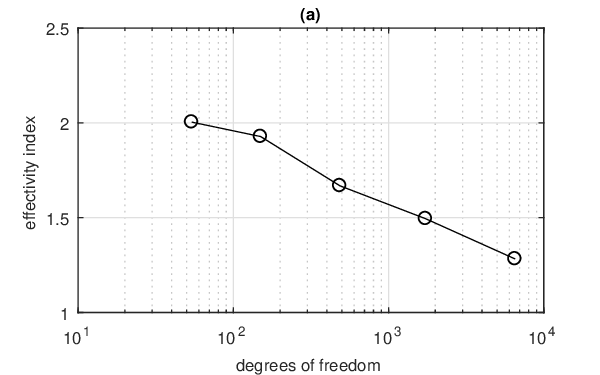}
\end{overpic}
\begin{overpic}[width=0.495\textwidth]{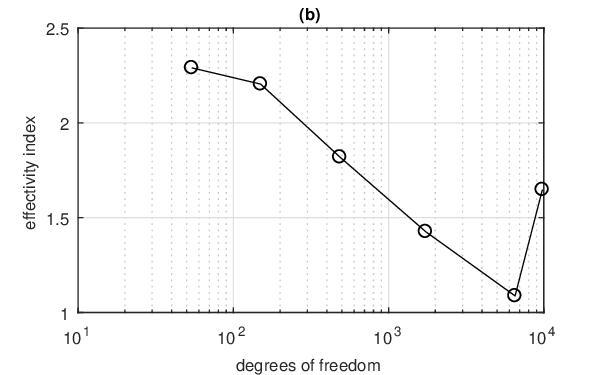}
\end{overpic}
\begin{overpic}[width=0.495\textwidth]{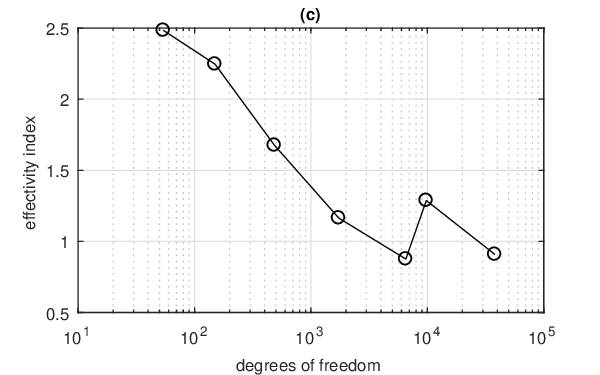}
\end{overpic}
\begin{overpic}[width=0.495\textwidth]{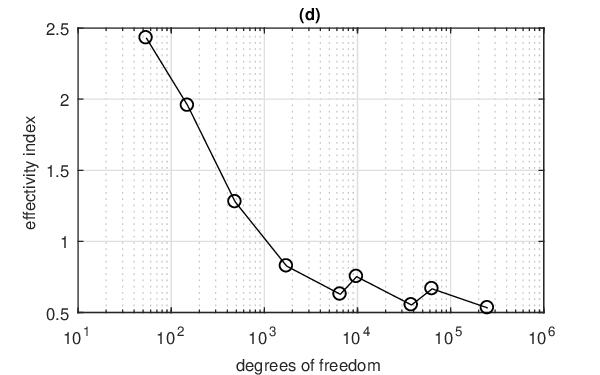}
\end{overpic}
\caption{The effectivity indices for the Galerkin solutions at each iteration of the adaptive algorithm
for the model problem~\eqref{paraPDE} with
$T(\bm{x}, \bm{y}) = a^2(\bm{x}, \bm{y})$ and the decomposition of $a(\bm{x}, \bm{y})$ as in Example~\ref{tp2} with $\bar{\sigma} = 2$:
(a) $\bar{\alpha} = 0.4$; (b) $\bar{\alpha} = 0.6$; (c) $\bar{\alpha} = 0.8$; (d) $\bar{\alpha} = 1$.}
\label{fig:eff_ind_2}
\end{figure}

\section{Concluding remarks}
\label{sec:con}

Adaptivity is a critical ingredient of effective algorithms for numerical solution of PDE problems with parametric or uncertain inputs.
In this paper, we consider a linear elliptic PDE with a generic parametric coefficient satisfying minimal assumptions that guarantee
well-posedness of the weak formulation in standard Lebesgue--Bochner spaces.
Building on earlier works for PDEs with affine-parametric representation of input data,
we have performed a posteriori error analysis of Galerkin approximations
and designed an adaptive solution algorithm for the considered problem.
An important contribution of this work is that it opens the possibility of solving elliptic PDE problems
with \emph{non-affine} parametric representations of input data using Galerkin approximations with rigorous error control,
thus, providing an effective alternative to traditional sampling techniques for such problems.
Furthermore, our proof of concept implementation and extensive numerical tests demonstrate
the effectiveness of our error estimation strategy and the practicality of the developed adaptive algorithm
for this class of parametric PDE problems.

\begin{appendices}
\crefalias{section}{appsec} 

\section{gPC expansion coefficients for parametric exponential and quadratic functions} \label{appA}


\fx{In this paper, we work with two forms of the diffusion coefficient $T(\bm{x}, \bm{y})$: $T(\bm{x}, \bm{y}) = \exp(a(\bm{x}, \bm{y}))$
and $T(\bm{x}, \bm{y}) = a^2(\bm{x}, \bm{y})$, where $a(\bm{x}, \bm{y})$ is given by~(\ref{kl}).}
\fx{For $T = \exp(a)$, we are able to separate the variables $y_m$.}
Specifically, for $M >1$, the integral in (\ref{coeexpansioncoe}) can be expressed as a product of 1D integrals as follows:
\begin{align}
\label{tgammaM}
t_{\bm{\gamma}}(\bm{x}) &=\exp(a_0(\bm{x})) \prod_{m=1}^M\int_{\Gamma_m} \exp(a_m(\bm{x}) y_m) p_{\gamma_m}^m(y_m) q_m(y_m) \dif y_m. 
\end{align}
The 1D integrals with respect to  $y_m$ in (\ref{tgammaM}) can be approximated numerically by using Gaussian quadrature.

For $T= a^2$, the infinite sum in (\ref{coeexpansion}) is naturally truncated
to a finite sum of terms indexed by $\bm{\gamma} \in \mathcal{P}_{M,2}$. 
Indeed, we have
\begin{align*}
T = a_0^2 + 2a_0\sum_{m=1}^M  a_m y_m + \sum_{m=1}^M a_m^2 y_m^2 + 2\sum_{m = 2}^M \sum_{n=1}^{m-1} a_m a_n y_m y_n
\end{align*}
and 
\begin{align*}
\label{tgammaMg1}
t_{\bm{\gamma}}(\bm{x}) = ~ & a_0^2\prod_{i=1}^M\int_{\Gamma_i} p_{\gamma_i} ^i q_i \dif y_i + 2 a_0 \sum_{m=1}^Ma_m\Bigg( \prod_{
\stackrel{i=1}{ i \neq m}}^M\int_{\Gamma_i} p_{\gamma_i} ^i q_i \dif y_i \Bigg)\int_{\Gamma_m} y_mp_{\gamma_m} ^m q_m \dif y_m \nonumber \\
&+ \sum_{m=1}^Ma_m^2\Bigg( \prod_{\stackrel{i=1}{i \neq m}}^M\int_{\Gamma_i} p_{\gamma_i} ^i q_i \dif y_i \Bigg)\int_{\Gamma_m} y_m^2p_{\gamma_m} ^m q_m \dif y_m \nonumber \\
&+2\sum_{m = 2}^M \sum_{n=1}^{m-1} a_m a_n\Bigg( \prod_{\stackrel{i=1}{i \neq m,n}}^M\int_{\Gamma_i} p_{\gamma_i} ^i q_i \dif y_i \Bigg)\int_{\Gamma_m} y_mp_{\gamma_m} ^m q_m \dif y_m \int_{\Gamma_n} y_n p_{\gamma_n}^n q_n \dif y_n.
\end{align*}
The orthogonality of $\{p_{n}^m\}_{n \in \mathbb{N}_0}$ gives the following conditions:
\begin{align*}
\int_{\Gamma_i} p_{\gamma_i} ^i q_i \dif y_i &= \begin{cases}
1 &\text{ for } \gamma_i = 0, \\
0 &  \text{ for } \gamma_i \neq 0,
\end{cases}
\\
\int_{\Gamma_i} y_ip_{\gamma_i} ^i q_i \dif y_i &= 0 \text{ for } \gamma_i >1, \quad
\int_{\Gamma_i} y_i^2 p_{\gamma_i} ^i q_i \dif y_i = 0 \text{ for } \gamma_i >2;
\end{align*}
this implies that $t_{\bm{\gamma}} = 0$
for $\bm{\gamma} \notin \mathcal{P}_{M,2}$.
If $M = 1$, then there are only three non-zero terms in (\ref{coeexpansion}) that are indexed by
$\bm{\gamma} \in\mathcal{P}_{1,2} \subset \mathbb{N}_0$ with the spatial expansion coefficients as follows:
\begin{align*}
t_{\bm{\gamma}}(\bm{x}) =
\begin{cases}
 a_0^2 + 2 a_0 a_1\int_{\Gamma_1} y_1 q_1(y_1) \dif y_1 + a_1^2\int_{\Gamma_1} y_1^2 q_1(y_1) \dif y_1 &
 \text{for }  \bm{\gamma} = (0), \\
 2 a_0 a_1\int_{\Gamma_1} y_1p_1^1(y_1) q_1(y_1) \dif y_1 + a_1^2\int_{\Gamma_1} y_1^2p_1^1(y_1) q_1(y_1) \dif y_1 &
 \text{for }  \bm{\gamma} = (1),\\
 a_1^2\int_{\Gamma_1} y_1^2p_2^1(y_1) q_1(y_1) \dif y_1 &
 \text{for }  \bm{\gamma} = (2).
\end{cases}
\end{align*}
If $M>1$,  then there are four types of non-zero items in (\ref{coeexpansion}) that are indexed by $\bm{\gamma} \in \mathcal{P}_{M,2}$
with the following spatial expansion coefficients:
\begin{enumerate}
\item[(i)] if $\bm{\gamma} = \bm{0}$, then
\begin{align*}
t_{\bm{\gamma}}(\bm{x}) &=   a_0^2 + 2 a_0 \sum_{m=1}^Ma_m\int_{\Gamma_m} y_m q_m \dif y_m + \sum_{m=1}^Ma_m^2\int_{\Gamma_m} y_m^2 q_m \dif y_m \nonumber \\
&\quad +2\sum_{m = 2}^M \sum_{n = 1}^{m-1} a_m a_n\int_{\Gamma_m} y_m q_m \dif y_m \int_{\Gamma_n} y_n q_n \dif y_n;
\end{align*}
\item[(ii)] if $\bm{\gamma}$ has only one non-zero element, $\gamma_j = 1$, $1 \leq j \leq M$, then
\begin{align*}
t_{\bm{\gamma}}(\bm{x}) &= 2 a_0 a_j\int_{\Gamma_j} y_j p_1 ^j q_j \dif y_j + a_j^2\int_{\Gamma_j} y_j^2 p_1 ^j q_j \dif y_j \nonumber \\
&\quad +2 \sum_{\stackrel{n = 1}{n \neq j}}^{M}a_j a_n\int_{\Gamma_j} y_j p_1 ^j q_j \dif y_j \int_{\Gamma_n} y_n q_n \dif y_n;
\end{align*}
\item[(iii)] if $\bm{\gamma}$ has only one non-zero element, $\gamma_j = 2$, $1 \leq j \leq M$, then
\begin{align*}
t_{\bm{\gamma}}(\bm{x}) = a_j^2 \int_{\Gamma_j} y_j^2 p_2 ^j q_j \dif y_j;
\end{align*}
\item[(iv)] if $\bm{\gamma}$ has only two non-zero elements, $\gamma_i = \gamma_j = 1$, $1 \leq i < j \leq M$, then
\begin{align*}
t_{\bm{\gamma}}(\bm{x}) = 2 a_i a_j \int_{\Gamma_i} y_i p_1 ^i q_i \dif y_i  \int_{\Gamma_j}y_j p_1 ^j q_j \dif y_j .
\end{align*}
\end{enumerate}
Thus, the index sets $\mathcal{N}(\mathcal{P}, \mathcal{P})$ in \eqref{fullGalerkin}, \eqref{eYPB0}, \eqref{eYPB1},
$\mathcal{N}(\mathcal{P}, \mathcal{Q})$ in \eqref{eXQB0}, \eqref{eXQB1},
and $\mathcal{N}(\mathcal{Q}, \mathcal{Q})$ in \eqref{eXQB1} are replaced by
$\mathcal{N}(\mathcal{P}, \mathcal{P})\cap \mathcal{P}_{M,2}$,
$\mathcal{N}(\mathcal{P}, \mathcal{Q})\cap \mathcal{P}_{M,2}$, and
$\mathcal{N}(\mathcal{Q}, \mathcal{Q})\cap \mathcal{P}_{M,2}$, respectively.

\end{appendices}

\bibliographystyle{siam}
\bibliography{ref} 

\begin{thebibliography}{10}

\bibitem{Ainsworth2000}
{\sc M.~Ainsworth and J.~T. Oden}, {\em A Posteriori Error Estimation in Finite
  Element Analysis}, John Wiley \& Sons, 2000.

\bibitem{2014-Bespalov-vol36}
{\sc A.~Bespalov, C.~E. Powell, and D.~Silvester}, {\em {Energy norm a
  posteriori error estimation for parametric operator equations}}, SIAM Journal
  on Scientific Computing, 36 (2014), pp.~A339--A363.

\bibitem{2019-Bespalov-p2359}
{\sc A.~Bespalov, D.~Praetorius, L.~Rocchi, and M.~Ruggeri}, {\em Convergence
  of adaptive stochastic {Galerkin} {FEM}}, SIAM Journal on Numerical Analysis,
  57 (2019), pp.~2359--2382.

\bibitem{2018-Bespalov-p243}
{\sc A.~Bespalov and L.~Rocchi}, {\em {Efficient adaptive algorithms for
  elliptic {PDEs} with random data}}, SIAM/ASA Journal on Uncertainty
  Quantification, 6 (2018), pp.~243--272.

\bibitem{2016-Bespalov-vol38}
{\sc A.~Bespalov and D.~Silvester}, {\em {Efficient adaptive stochastic
  Galerkin methods for parametric operator equations}}, SIAM Journal on
  Scientific Computing, 38 (2016), pp.~A2118--A2140.

\bibitem{doerfler}
{\sc W.~D{\"o}rfler}, {\em A convergent adaptive algorithm for {P}oisson's
  equation}, SIAM Journal on Numerical Analysis, 33 (1996), pp.~1106--1124.

\bibitem{2015-Eigel-p1367}
{\sc M.~Eigel, C.~Gittelson, C.~Schwab, and E.~Zander}, {\em {A convergent
  adaptive stochastic Galerkin finite element method with quasi-optimal spatial
  meshes}}, ESAIM Mathematical Modelling and Numerical Analysis, 49 (2015),
  pp.~1367--1398.

\bibitem{2014-Eigel-vol270}
{\sc M.~Eigel, C.~J. Gittelson, C.~Schwab, and E.~Zander}, {\em {Adaptive
  stochastic Galerkin {FEM}}}, Computer Methods in Applied Mechanics and
  Engineering, 270 (2014), pp.~247--269.

\bibitem{EigelMPS_ASG}
{\sc M.~Eigel, M.~Marschall, M.~Pfeffer, and R.~Schneider}, {\em Adaptive
  stochastic {G}alerkin {FEM} with for lognormal coefficients in hierarchical
  tensor representations}, {P}reprint 2515, WIAS, 2018.
\newblock \url{http://dx.doi.org/10.20347/WIAS.PREPRINT.2515}.

\bibitem{em16}
{\sc M.~Eigel and C.~Merdon}, {\em Local equilibration error estimators for
  guaranteed error control in adaptive stochastic higher-order {G}alerkin
  finite element methods}, SIAM/ASA Journal on Uncertainty Quantification, 4
  (2016), pp.~1372--1397.

\bibitem{2017-Eigel-p765}
{\sc M.~Eigel, M.~Pfeffer, and R.~Schneider}, {\em {Adaptive stochastic
  Galerkin FEM with hierarchical tensor representations}}, Numerische
  Mathematik, 136 (2017), pp.~765--803.

\bibitem{eijkhoutVass1991}
{\sc V.~Eijkhout and P.~Vassilevski}, {\em The role of the strengthened
  {C}auchy-{B}u\-nia\-kow\-ski\u\i -{S}chwarz inequality in multilevel
  methods}, SIAM Rev., 33 (1991), pp.~405--419.

\bibitem{2012-Ernst-p317}
{\sc O.~G. Ernst, A.~Mugler, H.-J. Starkloff, and E.~Ullmann}, {\em On the
  convergence of generalized polynomial chaos expansions}, ESAIM: Mathematical
  Modelling and Numerical Analysis, 46 (2012), pp.~317--339.

\bibitem{Gautschi2004}
{\sc W.~Gautschi}, {\em Orthogonal Polynomials: Computation and Approximation},
  Oxford University Press, Oxford, 2004.

\bibitem{gs91}
{\sc R.~G. Ghanem and P.~D. Spanos}, {\em Stochastic Finite Elements: a
  Spectral Approach}, Springer-Verlag, New York, 1991.

\bibitem{gittelson13}
{\sc C.~J. Gittelson}, {\em An adaptive stochastic {G}alerkin method for random
  elliptic operators}, Mathematics of Computation, 82 (2013), pp.~1515--1541.

\bibitem{2000-Pellissetti-p607}
{\sc M.~F. Pellissetti and R.~G. Ghanem}, {\em Iterative solution of systems of
  linear equations arising in the context of stochastic finite elements},
  Advances in Engineering Software, 31 (2000), pp.~607--616.

\bibitem{2009-Powell-p350}
{\sc C.~E. Powell and H.~C. Elman}, {\em Block-diagonal preconditioning for
  spectral stochastic finite-element systems}, IMA Journal of Numerical
  Analysis, 29 (2009), pp.~350--375.

\bibitem{2011-Schwab-p291}
{\sc C.~Schwab and C.~J. Gittelson}, {\em {Sparse tensor discretizations of
  high-dimen\-sional parametric and stochastic PDEs}}, Acta Numerica, 20
  (2011), pp.~291--467.

\bibitem{SIFISS}
{\sc D.~J. Silvester, A.~Bespalov, and C.~E. Powell}, {\em Stochastic {IFISS}
  {(S-IFISS)}, version $1.04$}, June 2017.
\newblock Available online at
  \url{http://www.manchester.ac.uk/ifiss/sifiss.html}.

\bibitem{2010-Ullmann-p923}
{\sc E.~Ullmann}, {\em {A Kronecker product preconditioner for stochastic
  Galerkin finite element discretizations}}, SIAM Journal on Scientific
  Computing, 32 (2010), pp.~923--946.

\bibitem{2002-Xiu-p619}
{\sc D.~Xiu and G.~E. Karniadakis}, {\em The {Wiener}--{Askey} polynomial chaos
  for stochastic differential equations}, SIAM Journal on Scientific Computing,
  24 (2002), pp.~619--644.

\end{thebibliography}

\end{document}